\documentclass[11pt,noinfoline]{article}
\RequirePackage[OT1]{fontenc}
\usepackage{amsthm}
\usepackage{amsmath}
\usepackage{amssymb}

\def\dec{\mathrm{dec}}
\newcommand{\p}{\mathbb{P}}
\newcommand{\N}{\mathbb{N}}
\newcommand{\E}{\mathbb{E}}
\newcommand{\Var}{\mathrm{Var}}
\newcommand{\R}{\mathbb{R}}

\newcommand{\degree}{{\rm deg}}
\newtheorem{theorem}{Theorem}
\newtheorem{lemma}{Lemma}
\newtheorem{cor}{Corollary}
\newtheorem{defi}{Definition}

\newcommand{\loglog}{{\rm loglog\,}}

\newcommand{\ii}{{\rm i}}
\newcommand{\jj}{{\rm j}}

\newcommand*{\ind}[1]{\mathbf{1}_{\{#1\}}}
\newcommand*{\Ind}[1]{\mathbf{1}_{#1}}
\newcommand{\xd}{\mathbf{X}^{\rm \dec}}
\newcommand{\xdi}{\mathbf{X}_\ii^{\rm dec}}
\newcommand{\xdij}{\mathbf{X}_{\ii_{J_j}}^{\rm dec}}
\newcommand{\xdik}{\mathbf{X}_{\ii_{K}}^{\rm dec}}
\newcommand{\xdii}{\mathbf{X}_{\ii_{I}}^{\rm dec}}
\newcommand{\xdiic}{\mathbf{X}_{\ii_{I^c}}^{\rm dec}}
\newcommand{\xdk}{\mathbf{X}_K}
\newcommand{\xdjz}{\mathbf{X}_{J_0}}
\newcommand{\edi}{\mathbf{\epsilon}_\ii^{\rm dec}}
\newcommand{\xui}{\mathbf{X}_\ii}
\newcommand{\eui}{\mathbf{\epsilon}_\ii}
\newcommand{\nodiag}[1]{I_#1^d}
\newcommand{\llog}{{\rm LL}}
\newcommand*{\nomax}[1]{\diamond#1}
\newcommand{\kimax}{{\rm k}_{I}}
\newcommand{\kiimax}{{\rm k}_{(J_i\backslash J)^c}}
\newcommand{\kijmax}{{\rm k}_{(J_j\backslash J)^c}}

\title{The LIL for $U$-statistics in Hilbert spaces}

\author{{Rados{\l}aw Adamczak}\thanks{Institute of Mathematics, Polish Academy of Sciences, Warsaw, Poland.
Email: R.Adamczak@impan.gov.pl. Research partially supported by
MEiN Grant 2 PO3A 019 30.}\quad {Rafa{\l}
Lata{\l}a}\thanks{Institute of Mathematics, Warsaw University,
Warsaw, Poland. Email: RLatala@mimuw.edu.pl. Research partially
supported by MEiN Grant 1 PO3A 012 29.}}

\begin{document}
\maketitle

\begin{abstract}
We give necessary and sufficient conditions for the (bounded) law
of the iterated logarithm for $U$-statistics in Hilbert spaces. As
a tool we also develop moment and tail estimates for canonical
Hilbert-space valued $U$-statistics of arbitrary order, which are
of independent interest.


\end{abstract}



\section{Introduction}
In the last two decades we have witnessed a rapid development in the asymptotic theory of $U$-statistics, boosted
by the introduction of the so called 'decoupling' techniques (see \cite{dlp2, dlp1,dlp3}), which allow to treat $U$-statistics conditionally
as sums of independent random variables. This approach yielded better understanding of $U$-statistics versions of the classical limit theorems
of probability. Necessary and sufficient conditions were found for the strong law of large numbers \cite{LZ}, the central limit theorem \cite{RV,GZ2}
and the law of the iterated logarithm \cite{GLKZ,AmcLat}. Also some sharp exponential inequalities for canonical $U$-statistics
have been found \cite{GLZ, Amc, H}. Analysis of the aforementioned results shows an interesting phenomenon. Namely, the
natural counterparts of the necessary and sufficient conditions for sums of i.i.d. random variables ($U$-statistics of degree 1), remain sufficient
for $U$-statistics of arbitrary degree, but with an exception for the CLT, they cease to be necessary. The correct conditions turn out to be much
more involved and are expressed for instance in terms of convergence of some series (LLN) or as growth conditions for some functions (LIL).

A natural problem is an extension of the above results to the infinite-dimensional setting. There has been some progress in this direction, and partial answers have
been found, usually under the assumption on the geometrical structure of the space in which the values of a $U$-statistic are taken. In general however the picture
is far from being complete and the necessary and sufficient conditions are known only in the case of the CLT for Hilbert space valued $U$-statistics (see \cite{dlp2, GZ2}
for the proof of sufficiency in type 2 spaces and necessity in cotype 2 spaces respectively).

In this article we generalize to separable Hilbert spaces the results from \cite{AmcLat} on necessary and sufficient conditions for the LIL for
real valued $U$-statistics. The conditions are expressed only in terms of the $U$-statistic kernel and the distribution of the underlying i.i.d.
sequence and can be also considered a generalization of results from \cite{GKZ}, where the LIL for i.i.d. sums in Hilbert spaces was characterized.
We consider only the bounded version of the LIL and do not give the exact value of the $\limsup$ nor determine the limiting set.
Except for the classical case of sums of i.i.d. random variables, the problem of finding the $\limsup$ is at the moment
open even in the one dimensional case (see \cite{AG1, dlp2, KLOZ} for some partial results)
and the problem of the geometry of the limiting set and the compact LIL is solved only under suboptimal integrability conditions \cite{AG1}.

The organization of the paper is as follows. First, in Section \ref{nierownosci} we prove sharp exponential inequalities
for canonical $U$-statistics, which generalize the results of \cite{Amc, GLZ} for the real-valued case. Then, after recalling
some basic facts about the LIL we give necessary and sufficient condition for the LIL for decoupled, canonical $U$-statistics (Theorem \ref{decoupled_lil_theorem}).
The quite involved proof is given in the two subsequent sections. Finally we conclude with our main result (Theorem \ref{undecoupled_lil_theorem}), which gives a characterization
of the LIL for undecoupled $U$-statistics and follows quite easily from Theorem \ref{decoupled_lil_theorem} and
the one dimensional result.

\section{Notation}
For an integer $d$, let $(X_i)_{i\in \N}$, $(X_i^{(k)})_{i\in \N,
1 \le k \le d}$ be independent random variables with values in a Polish
space $\Sigma$, equipped with the Borel $\sigma$-field
$\mathcal{F}$. Let also $(\varepsilon_i)_{i\in \N}$, $(\varepsilon_i^{(k)})_{i\in \N,
1 \le k \le d}$ be independent Rademacher variables, independent of $(X_i)_{i\in \N}$, $(X_i^{(k)})_{i\in \N,
1 \le k \le d}$.

Consider moreover measurable functions $h_\ii\colon
\Sigma^d \to H$, where $(H,|\cdot|)$ is a separable Hilbert space (we will denote both
the norm in $H$ and the absolute value of a real number by $|\cdot|$, the context will however
prevent ambiguity).

To shorten the notation, we will use the following convention. For
$\ii = (i_1,\ldots,i_d) \in \{1,\ldots,n\}^d$ we will write $\xui$
(resp. $\xdi$) for $(X_{i_1},\ldots,X_{i_d})$, (resp.
$(X_{i_1}^{(1)},\ldots,X_{i_d}^{(d)})$) and $\eui$ (resp. $\edi$)
for the product
$\varepsilon_{i_1}\cdot\ldots\cdot\varepsilon_{i_d}$ (resp.
$\varepsilon_{i_1}^{(1)}\cdot\ldots\cdot\varepsilon_{i_d}^{(d)}$),
the notation being thus slightly inconsistent, which however
should not lead to a misunderstanding. The $U$-statistics will
therefore be denoted
\begin{align*}
&\sum_{\ii \in \nodiag{n}} h_\ii(\xui) &\textrm{(an undecoupled $U$-statistic)}\\
&\sum_{|\ii| \le n} h_\ii(\xdi) &\textrm{(a decoupled $U$-statistic)}\\
&\sum_{\ii \in \nodiag{n}} \eui h_\ii(\xui) &\textrm{(an undecoupled randomized $U$-statistic)}\\
&\sum_{|\ii| \le n} \edi h_\ii(\xdi) &\textrm{(a decoupled randomized
$U$-statistic)},
\end{align*}
where 
\begin{align*}
|\ii| & = \max_{k=1,\ldots,d} i_k,\\
\nodiag{n} &= \{\ii\colon |\ii|\le n, i_j \neq i_k\;\textrm{for}\;
j\neq k\}.
\end{align*}
Since in this notation $\{1,\ldots,d\} = I_d^1$ we will write
\begin{displaymath}
I_d = \{1,2,\ldots,d\}.
\end{displaymath}

Throughout the article we will write $L_d, L$ to denote
constants depending only on $d$ and
universal constants respectively. In all those cases the values of
a constant may differ at each occurrence.

For $I \subseteq I_d$, we will write $\E_I$ to denote integration
with respect to variables $(X_{i}^{(j)})_{i\in \N,j\in I}$. We
will consider mainly canonical (or completely degenerated)
kernels, i.e. kernels $h_\ii$, such that for all $j \in I_d$,
$\E_j h_\ii(\xdi) = 0$ a.s.

\section{Moment inequalities for $U$-statistics in Hilbert space \label{nierownosci}}

In this section we will present sharp moment and tail inequalities
for Hilbert space valued $U$-statistics, which in the sequel will
constitute an important ingredient in the analysis of the LIL.
These estimates are a natural generalization of inequalities for
real valued U-statistics presented in \cite{Amc}.

Let us first introduce some definitions.

\begin{defi}
For a nonempty, finite set $I$ let $\mathcal{P}_{I}$ be the family
consisting of all partitions $\mathcal{J} = \{J_1,\ldots,J_k\}$ of
$I$ into nonempty, pairwise disjoint subsets. Let us also define
for $\mathcal{J}$ as above $\degree(\mathcal{J}) = k$.
Additionally let $\mathcal{P}_\emptyset = \{\emptyset\}$ with
$\degree(\emptyset) = 0$.
\end{defi}

\begin{defi}\label{deterministic_norms}
For a nonempty set $I \subseteq I_d$ consider $\mathcal{J} =
\{J_1,\ldots,J_k\} \in \mathcal{P}_{I}$. For an array
$(h_\ii)_{\ii \in I_n^d}$ of H-valued kernels and fixed value of
$\ii_{I^c}$, define
\begin{align*}
\|(h_\ii)_{\ii_{I}}\|_\mathcal{J} = \sup\Big\{&\Big|\sum_{\ii_{I}}\E_I[
h_{\ii}(\xdi)\prod_{j=1}^{\degree(\mathcal{J})}
f_{\ii_{J_j}}^{(j)}(\xdij)]\Big| \colon f_{\ii_{J_j}}^{(j)}\colon \Sigma^{J_j} \to \R\\
&\E \sum_{\ii_{J_j}} |f_{\ii_{J_j}}^{(j)}(\xdij)|^2 \le 1 \; {\rm
for }\; j = 1,\ldots,\degree(\mathcal{J})\Big\}.
\end{align*}
Let moreover $\|(h_\ii)_{\ii_\emptyset}\|_{\emptyset}= |h_\ii|$.
\end{defi}

\paragraph{Remark} It is worth mentioning that for $I = I_d$,
$\|\cdot\|_\mathcal{J}$ is a deterministic norm, whereas for
$I\subsetneq I_d$ it is a random variable, depending on
$\mathbf{X}_{\ii_{I^c}}^{\rm dec}$.

Quantities given by the above definition suffice to obtain
precise moment estimates for real valued $U$-statistics.
However, to bound the moments of
$U$-statistics with values in general Hilbert spaces, we will need
to introduce one more definition.

\begin{defi}\label{determnistic_norms_Hilbert}
For nonempty sets $K\subseteq I \subseteq I_d$ consider
$\mathcal{J} = \{J_1,\ldots,J_k\} \in \mathcal{P}_{I\backslash
K}$. For an array $(h_\ii)_{\ii \in I_n^d}$ of H-valued kernels
and fixed value of $\ii_{I^c}$, define
\begin{align*}
\|(h_\ii)_{\ii_{I}}\|_{K,\mathcal{J}} =
\sup\Big\{&|\sum_{\ii_{I}}\E_I [\langle
h_{\ii}(\xdi),g_{\ii_K}(\xdik)\rangle\prod_{j=1}^{\degree(\mathcal{J})}
f_{\ii_{J_j}}^{(j)}(\xdij)]| \colon \\
&f_{\ii_{J_j}}^{(j)} \colon \Sigma^{J_j} \to \R,\; g_{\ii_K}\colon
\Sigma^K \to H\;,\E \sum_{\ii_{K}}|g_{\ii_K}(\xdik)|^2 \le 1\\
&\E \sum_{\ii_{J_j}} |f_{\ii_{J_j}}^{(j)}(\xdij)|^2 \le 1 \; {\rm
for }\; j = 1,\ldots,\degree(\mathcal{J})
\Big\}.
\end{align*}
\end{defi}

\paragraph{Remark} One can see that the only difference between
the above definition and Definition \ref{deterministic_norms} is
that the latter distinguishes one set of coordinates and allows
functions corresponding to this set to take values in $H$.
Moreover, since the norm in $H$ satisfies $|\cdot| =
\sup_{|\phi|\le 1}\langle\phi,\cdot\rangle$, we can treat
Definition \ref{deterministic_norms} as a counterpart of
Definition \ref{determnistic_norms_Hilbert} for $K = \emptyset$.
We will use this convention to simplify the statements of the
subsequent theorems. Thus, from now on, we will write
\begin{displaymath}
\|\cdot\|_{\emptyset,\mathcal{J}} := \|\cdot\|_{\mathcal{J}}.
\end{displaymath}

\paragraph{Example} For $d=2$ and $I = \{1,2\}$, the above definition gives
\begin{align*}
\|(h_{ij}(X_i,Y_j))_{i,j}\|_{\emptyset,\{\{1,2\}\}} & = \sup\Big\{\Big|\E \sum_{ij} h_{ij}(X_i,Y_j)f_{ij}(X_i,Y_j)\Big|\colon \\
&\phantom{xxxxxx} \E \sum_{ij} f(X_i,Y_j)^2\le 1\Big\}\\
&=\sup_{\phi\in H,|\phi|\le 1}\sqrt{\E \sum_{ij} \langle \phi,h_{ij}(X_i,Y_j)\rangle^2},\\
\|(h_{ij}(X_i,Y_j))_{i,j}\|_{\emptyset, \{\{1\}\{2\}\}} & = \sup\Big\{\Big|\E \sum_{ij} h_{ij}(X_i,Y_j)f_i(X_i)g_j(Y_j)\Big|\colon \\
&\phantom{xxxxxx}  \sum_i \E f(X_i)^2, \sum_j \E g(Y_j)^2\le 1\Big\},\\
\|(h_{ij}(X_i,Y_j))_{i,j}\|_{\{1\},\{\{2\}\}} & = \sup\Big\{\E \sum_{ij} \langle f_i(X_i),h_{ij}(X_i,Y_j)\rangle g_j(Y_j)\colon \\
&\phantom{xxxxxx} \E \sum_i |f(X_i)|^2,\E \sum_j g(Y_j)^2\le 1\Big\},\\
\|(h_{ij}(X_i,Y_j))_{i,j}\|_{\{1,2\},\emptyset} & = \sup\Big\{\E \sum_{ij}\langle f_{ij}(X_i,Y_j),h_{ij}(X_i,Y_j)\rangle\colon\\
&\phantom{xxxxxx} \E \sum_{ij} |f(X_i,Y_j)|^2\le 1\Big\}\\
& = \sqrt{\sum_{i,j}\E|h_{ij}(X_i,Y_j)|^2}.
\end{align*}

We can now present the main result of this section.

\begin{theorem}\label{moment_estimates} For any array of $H$-valued,
completely degenerate kernels $(h_\ii)_\ii$ and any $p \ge 2$, we
have
\begin{align*}
\E\big|\sum_\ii h(\xdi)\big|^p \le L_d^p\Big(\sum_{K\subseteq I\subseteq
I_d}\sum_{\mathcal{J}\in \mathcal{P}_{I\backslash K}} p^{p(\#I^c +
\deg\mathcal{J}/2)}\E_{I^c}\max_{\ii_{I^c}}\|(h_\ii)_{\ii_I}\|_{K,\mathcal{J}}^p\Big).
\end{align*}
\end{theorem}

The proof of the above theorem proceeds along the lines of arguments presented in \cite{Amc,GLZ}. In particular we will need the following moment estimates for suprema of
empirical processes \cite{GLZ}.
\begin{lemma}[{\cite[Proposition 3.1]{GLZ}, see also \cite[Theorem 12]{BBLM}}]
\label{Talagrand} Let $X_1,\ldots,X_n$ be independent random
variables with values in $(\Sigma,\mathcal{F})$ and $\mathcal{T}$
be a countable class of measurable real functions on $\Sigma$,
such that for all $f \in \mathcal{T}$ and $i \in I_n$, $\E f(X_i)
= 0$ and $\E f(X_i)^2 < \infty$.  Consider the random variable $S
:= \sup_{f \in \mathcal{T}} |\sum_i f(X_i)|$. Then for all $p \ge
1$,
\begin{displaymath}
\E S^p \le L^p\left[ (\E S)^p + p^{p/2}\sigma^p +
p^p\E\max_i\sup_{f \in \mathcal{T}} |f(X_i)|^p\right],
\end{displaymath}
where
\begin{displaymath}
\sigma^2 = \sup_{f \in \mathcal{T}} \sum_i \E f(X_i)^2.
\end{displaymath}
\end{lemma}

We will also need the following technical lemma.
\begin{lemma}[Lemma 5 in \cite{Amc}]\label{sumy_na_maxima}
For $\alpha > 0$ and arbitrary nonnegative kernels $g_\ii \colon
\Sigma^d \to \R_+$ and $p > 1$ we have
\begin{displaymath}
p^{\alpha p}\sum_{\ii} \E g_\ii^p \le L_d^p p^{\alpha
d}\left[p^{\alpha p}\E\max_{\ii} g_\ii^p + \sum_{I \subsetneq
\{1,\ldots,d\}} p^{\#I p}\E_I \max_{\ii_I}(\sum_{\ii_{I^c}}
\E_{I^c} g_\ii)^p \right].
\end{displaymath}
\end{lemma}

Before stating the next lemma, let us introduce some more
definitions, concerning $\mathcal{J}$--norms of deterministic
matrices

\begin{defi} Let $(a_\ii)_{\ii \in I_n^d}$ be a $d$-indexed array of real numbers.
For $\mathcal{J} = \{J_1,\ldots,J_k\} \in \mathcal{P}_{I_d}$
define
\begin{displaymath}
\|(a_\ii)_\ii\|_{\mathcal{J}} =\sup\Big\{\sum_{\ii}
a_{\ii}x_{\ii_{J_{1}}}^{(1)}\cdots x_{\ii_{J_{k}}}^{(k)}\colon
\sum_{\ii_{J_{1}}}(x_{\ii_{J_{1}}}^{(1)})^{2}\leq 1,\ldots,
\sum_{\ii_{J_{k}}}(x_{\ii_{J_{k}}}^{(k)})^{2}\leq 1\Big\}.
\end{displaymath}
\end{defi}

We will also need
\begin{defi}
For $\ii \in \N^{d-1}\times I_n$ let $a_\ii \colon \Sigma \to \R$
be measurable functions and $Z_1,\ldots,Z_n$ be independent random variables with values in $\Sigma$.
For a partition $\mathcal{J} =
\{J_1,\ldots,J_k\} \in \mathcal{P}_{I_d}$ ($d \in J_1$), let us
define
\begin{align*}
\|(a_\ii(Z_{i_d}))_\ii\|_{\mathcal{J}}
=\sup\Big\{&\sqrt{\sum_{\ii_{J_1}}\E \Big(\sum_{\ii_{I_d\backslash
J_1}} a_{\ii}(Z_{i_d})x_{\ii_{J_{2}}}^{(2)}
\cdots x_{\ii_{J_{k}}}^{(k)}\Big)^2}\colon\\
& \sum_{\ii_{J_{2}}}(x_{\ii_{J_{2}}}^{(2)})^{2}\leq 1,\ldots,
\sum_{\ii_{J_{k}}}(x_{\ii_{J_{k}}}^{(k)})^{2}\leq 1\Big\}.
\end{align*}
\end{defi}

\paragraph{Remark} All the definitions of norms presented so far,
seem quite similar and indeed they can be all interpreted as
injective tensor-product norms on proper spaces. We have decided
to introduce them separately by explicit formulas, because this
form appears in our applications.

The next lemma is crucial for obtaining moment inequalities for canonical real-valued
$U$-statistics of order greater than 2. In the context of $U$-statistics in Hilbert spaces we will need it already
for $d=2$.

\begin{lemma}[Theorem 5 in \cite{Amc}]\label{random_multimatrix}
Let $Z_1,\ldots,Z_n$ be independent random variables with values
in $(\Sigma,\mathcal{F})$. For $\ii \in \N^{d-1}\times I_n$ let
$a_\ii \colon \Sigma \to \R$ be measurable functions, such that
$\E_{Z}a_\ii(Z_{i_d}) = 0$. Then, for all $p \ge 2$ we have
\begin{align*}
\E\|(\sum_{i_d} a_{\ii}(Z_{i_d})&)_{\ii_{I_{d-1}}}\|  \le
L_d\sum_{\mathcal{J} \in \mathcal{P}_{I_d}}p^{(1 +
\deg{(\mathcal{J})} -
d)/2}\|(a_\ii(Z_{i_d}))_\ii\|_\mathcal{J} \\
&+ L_d\sum_{J\in \mathcal{P}_{I_{d-1}}}p^{1 + (1 +
\deg(\mathcal{J}) -
d)/2}\sqrt{\E\max_{i_d}\|(a_\ii(Z_{i_d}))_{\ii_{I_{d-1}}}\|_\mathcal{J}^2},
\end{align*}
where $\|\cdot\|$ denotes the norm of a $(d-1)$-indexed matrix,
regarded as a $(d-1)$-linear operator on $(l_2)^{d-1}$ (thus the
$\|\cdot\|_{\{1\}\ldots\{d-1\}}$--norm in our notation).
\end{lemma}

To prove Theorem \ref{moment_estimates}, we will need to adapt the above lemma to be able to
bound the $(K,\mathcal{J})$-norms of sums of independent
kernels.

\begin{defi}
We define a partial order $\prec$ on $P_I$ as
\begin{displaymath}
\mathcal{I} \prec \mathcal{J}
\end{displaymath}
if and only if for all $I \in \mathcal{I}$, there exists $J\in\mathcal{J}$, such that $I \subseteq J$.
\end{defi}

\begin{lemma} \label{crucial_lemma}Assume that $\sum_\ii \E|h_\ii(\xdi)|^2 < \infty$.
Then for any $K\subseteq I_{d-1}$ and $\mathcal{J} =
\{J_1,\ldots,J_k\} \in \mathcal{P}_{I_{d-1}\backslash K}$ and
all $p\ge 2$,
\begin{align}
\label{awful_formula}
\E_{d} &\|(\sum_{i_d}
h_\ii(\xdi))_{\ii_{I_{d-1}}}\|_{K,\mathcal{J}} \\
& \le L_d\Bigg(\sum_{\stackrel{K\subseteq L \subseteq I_{d},\;
\mathcal{K}\in\mathcal{P}_{I_d\backslash
L}\colon}{\mathcal{J}\cup\{K,\{d\}\}\prec \mathcal{K}\cup\{L\}}}
p^{(\deg\mathcal{K}-\deg\mathcal{J})/2}\|(h_\ii)_{\ii_{I_d}}\|_{L,\mathcal{K}}\nonumber\\
&+\sum_{\stackrel{K\subseteq L\subseteq
I_{d-1},\;\mathcal{K}\in\mathcal{P}_{I_{d-1}\backslash
L}\colon}{\mathcal{J}\cup\{K\}\prec
\mathcal{K}\cup\{L\}}}p^{1+(\deg\mathcal{K}-\deg\mathcal{J})/2}\sqrt{\E_d\max_{i_d}\|(h_\ii)_{\ii_{I_{d-1}}}\|_{L,\mathcal{K}}^2}\Bigg)\nonumber.
\end{align}
\end{lemma}
\paragraph{Remark} In the above lemma we slightly abuse the notation, by identifying for $K = \emptyset$ the
partition $\{\emptyset\}\cup\mathcal{J}$ with $\mathcal{J}$.

\paragraph{}Given Lemma \ref{random_multimatrix}, the proof of Lemma \ref{crucial_lemma} is not complicated, the main idea is just
a change of basis, however due to complicated notation it is quite difficult to write
it directly. We find it more convenient to write the proof in terms of tensor products of
Hilbert spaces.

Let us begin with a classical fact.

\begin{lemma} \label{iloczyn_tensorowy_Hilbert} Let $H$ be a separable Hilber space and
$X$ a $\Sigma$-valued random variable. Then
$H\otimes L^2(X) \simeq L^2(X,H)$, where $L^2(X,H)$ is the space of square integrable random variables
of the form $f(X)$,
$f\colon \Sigma \to H$-measurable. With the above identification,
for $h\in H$, $f(X)\in L^2(X)$, we have $h\otimes f(X) = hf(X) \in
L^2(X,H)$.
\end{lemma}

\begin{proof}[Proof of Lemma \ref{crucial_lemma}]
To avoid problems with notation, which would lengthen an intuitively easy proof,
we will omit some technical details, related to obvious identification of some tensor product
of Hilbert spaces (in the spirit of Lemma \ref{iloczyn_tensorowy_Hilbert}). Similarly, when considering
linear functionals on a space, which can be written as a tensor product in several ways, we will
switch to the most convenient notation, without further explanations.

Let
\begin{displaymath}
H_0 = H\otimes\big[\otimes_{l\in K}(\oplus_{i=1}^n
L^2(X_{i}^{(l)})] \simeq \oplus_{|\ii_K|\le n} L^2(\xdik,H)
\end{displaymath}
and, for $j=1,\ldots,k$,
\begin{displaymath}
H_i = \otimes_{l\in J_j}(\oplus_{i=1}^n L^2(X_{i}^{(l)})) \simeq
\oplus_{|\ii_{J_j}|\le n} L^2(\xdij).
\end{displaymath}

In the case $K=\emptyset$, we have (using the common convention for empty products)
 $H_0 \simeq H$.

For $i_d = 1,\ldots, n$ and fixed value of $X_{i_d}^{(d)}$,
let $A_{i_d}$ be a linear functional on $\tilde{H} =
\oplus_{|\ii_{I_{d-1}}|\le n}L^2(\mathbf{X}_{\ii_{I_{d-1}}}^{\rm
dec},H)\simeq \otimes_{j=0}^k H_k$, given by
$(h_\ii(\xdi))_{|\ii_{I_{d-1}}|\le n} \in \tilde{H}$, with the formula
\begin{align*}
A_{i_d}((g_{\ii_{I_{d-1}}}(\mathbf{X}_{\ii_{I_{d-1}}}^{\rm
dec}))_{\ii_{I_{d-1}}}) &= \langle
(g_{\ii_{I_{d-1}}}(\mathbf{X}_{\ii_{I_{d-1}}}^{\rm
dec}))_{\ii_{I_{d-1}}}, (h_\ii(\xdi))_{\ii_{I_{d-1}}}\rangle_{\tilde{H}}\\
&= \sum_{|\ii_{I_{d-1}}|\le n}\E_{\{1,\ldots,d-1\}}\langle
g_{\ii_{I_{d-1}}}(\mathbf{X}_{\ii_{I_{d-1}}}^{\rm dec}),
h_\ii(\xdi)\rangle_H.
\end{align*}
 As functions of
$X_{i_d}^{(d)}$, $A_{i_d}= A_{i_d}(X_{i_d}^{(d)})$ are
independent random linear functionals. Thus they determine
also random  $(k+1)$-linear functionals on $\oplus_{j=0}^k H_k$,
given by
\begin{displaymath}
(h_0,h_1,\ldots,h_k) \mapsto A_{i_d}(h_0\otimes
h_1\otimes\ldots\otimes h_k).
\end{displaymath}
If we denote by $\|\cdot\|$ the norm of a
$(k+1)$-linear functional, the left hand-side of (\ref{awful_formula}), can be written as

\begin{displaymath}
\E\big\|\sum_{i_d=1}^n A_{i_d}(X_{i_d}^{(d)})\big\|.
\end{displaymath}
Moreover, denoting by $\|A_{i_d}\|_{HS}$ the norm of $A_{i_d}$ seen as
a linear operator on $\otimes_{j=0}^k H_j$ (by analogy with the Hilbert-Schmidt norm of a matrix), we have
\begin{displaymath}
\sum_{i_d=1}^n \E \|A_{i_d}(X_{i_d}^{(d)})\|^2_{HS} =
\|(h_\ii)_\ii\|_{I_d,\emptyset}^2 < \infty,
\end{displaymath}
so
the sequence $A_{i_d}(X_{i_d}^{(d)})$, determines
a linear functional $A$ on $\tilde{H}\otimes
[\oplus_{i_d=1}^n L^2(X_{i_d}^{(d)})] \simeq \oplus_{|\ii|\le
n}L^2(\xdi,H) \simeq \oplus_{i_d=1}^n
L^2(X_{i_d}^{(d)},\tilde{H})$, given by the formula
\begin{displaymath}
A(g_1(X_1^{(d)}),\ldots,g_n(X_n^{(d)})) = \sum_{i_d=1}^n \E[
A_{i_d}(X_{i_d}^{(d)})(g_{i_d}(X_{i_d}^{(d)}))].
\end{displaymath}
It is easily seen, that if we interpret the domain of this functional as
 $\oplus_{|\ii|\le n}L^2(\xdi,H)$, then
it corresponds to the multimatrix $(h_\ii(\xdi))_\ii$.

Let us now introduce the following notation, consistent with the definition of
$\|\cdot\|_\mathcal{J}$.  If $T$ is a linear functional on
$\otimes_{j=0}^m E_j$ for some Hilbert spaces $E_j$, and
$\mathcal{I} = \{L_1,\ldots,L_r\}\in \mathcal{P}_{I_m\cup\{0\}}$,
then let $\|T\|_{\mathcal{I}}$ denote the norm of $T$ as a
$r$-linear functional on $\oplus_{i=1}^r[\otimes_{j\in L_i} E_j]$, given by

\begin{displaymath}
(e_1,\ldots,e_r) \mapsto T(e_1\otimes\ldots\otimes e_r).
\end{displaymath}

Now, denoting $H_{k+1}= \oplus_{i_d=1}^n L^2(X_{i_d}^{(d)})$,
we can apply the above definition to $\tilde{H}\otimes
[\oplus_{i_d=1}^n L^2(X_{i_d}^{(d)})] \simeq \otimes_{j=0}^{k+1} H_j$
and use Lemma \ref{random_multimatrix} to obtain
\begin{align}\label{forma_tensorowa}
\E\big\|\sum_{i_d=1}^n A_{i_d}(X_{i_d}^{(d)})\big\| \le&
L_d\sum_{\mathcal{I} \in \mathcal{P}_{I_{k+1}\cup\{0\}}}p^{(1 +
\deg{(\mathcal{I})} - (k+2))/2}\|A\|_\mathcal{I}  \nonumber\\
&+ L_d\sum_{\mathcal{I}\in \mathcal{P}_{I_{k}\cup\{0\}}}p^{1 + (1 +
\deg(\mathcal{I}) -
(k+2))/2}\sqrt{\E\max_{i_d}\|A_{i_d}(X_{i_d}^{(d)})\|_\mathcal{I}^2}.
\end{align}

This inequality is just the statement of the Lemma, which follows from
,,associativity'' of the tensor product and its ,,distributivity''
with respect to the simple sum of Hilbert spaces. Indeed, denoting
$J_{k+1} = \{d\}$, we have for $0\notin L_i$ and $U=
\bigcup_{j\in L_i} J_j$,
\begin{align*}
\otimes_{j \in L_i} H_j &\simeq \otimes_{j\in L_i}\otimes_{l\in
J_j}(\oplus_{s=1}^n L^2(X_{s}^{(l)}))\simeq \otimes_{l\in U}
(\oplus_{s=1}^n L^2(X_{s}^{(l)})) \simeq \oplus_{|\ii_U|\le n}
L^2(\mathbf{X}_{\ii_{U}}^{\rm dec}).
\end{align*}
Similarly, if $0 \in L_i$,
\begin{align*}
\otimes_{j \in L_i} H_j &\simeq [\oplus_{|\ii_K|\le n}
L^2(\xdik,H)]\times [\otimes_{0\neq j\in L_i}\otimes_{l\in
J_j}(\oplus_{s=1}^n L^2(X_{s}^{(l)}))]\\
& \simeq \oplus_{|\ii_U|\le n} L^2(\mathbf{X}_{\ii_{U}}^{\rm
dec},H),
\end{align*}
where $U = (\bigcup_{0\neq j\in L_i} J_j)\cup K$. Using the fact
that for fixed $X_{i_d}^{(d)}$, $A_{i_d}$ corresponds
to the multimatrix $(h_\ii(\xdi))_{|\ii_{I_{d-1}}|\le n}$, and $A$ corresponds
to $(h_\ii(\xdi))_{|\ii|\le n}$, we can see, that each summand
$\|\cdot\|_{\mathcal{I}}$ on the right hand side of
(\ref{forma_tensorowa}) is equal to some summand
$\|\cdot\|_{L,\mathcal{K}}$ on the right hand side of (\ref{awful_formula}). Informally speaking
and abusing slightly the notation (in the case $K=
\emptyset$), we ,,merge'' the elements of the partition
$\{\{d\},J_1,\ldots,J_k,K\}$ or $\{J_1,\ldots,J_k,K\}$ in a way described by the partition
 $\mathcal{I}$, thus obtaining the partition
$\{L\}\cup\mathcal{K}$, where $L$ is the set corresponding in the new partition
to the set $L_i \in \mathcal{I}$, containing $0$
(in particular, if $K=\emptyset$ and $\{0\} \in
\mathcal{I}$, then $L = \emptyset$). Let us also notice, that
$\deg(\mathcal{I}) = \deg(\mathcal{K}) + 1$, hence
\begin{displaymath}
1 + \deg(\mathcal{I}) - (k+2) = \deg(\mathcal{K}) -
\deg(\mathcal{J}),
\end{displaymath}
which shows, that also the powers of $p$ on the right hand sides of (\ref{awful_formula})
and (\ref{forma_tensorowa}) are the same, completing the proof.
\end{proof}

\begin{proof}[Proof of Theorem \ref{moment_estimates}]

For $d=1$, the theorem is an obvious consequence of Lemma
\ref{Talagrand}. Indeed, since $|\cdot| = \sup_{|\phi|\le
1}|\phi(\cdot)|$, and we can restrict the supremum to a countable
set of functionals, we have
\begin{align*}
\E|\sum_i h_i(X_i)|^p &\le L^p\big( (\E|\sum_i h_i(X_i)|)^p +
p^{p/2}\sup_{|\phi|\le 1}(\sum_i \E\langle
\phi,h_i(X_i)\rangle^2)^{p/2}\\
 &+ p^p\E\max_i|h_i(X_i)|^p\big).
\end{align*}

But $\E|\sum_i h_i(X_i)| \le \sqrt{\E|\sum_i h_i(X_i)|^2} =
\sqrt{\sum_i \E|h_i(X_i)|^2} = \|(h_i)_i\|_{\{1\},\emptyset}$ and
we also have $\sup_{|\phi|\le 1}(\sum_i \E\langle
\phi,h_i(X_i)\rangle^2)^{1/2} = \|(h_i)_i\|_{\emptyset,\{1\}}$ and $\max_i |h_i(X_i)| =
\max_i \|h_i\|_{\emptyset,\emptyset}$.

We will now proceed by induction with respect to $d$. Assume that
the theorem is true for all integers smaller than $d\ge 2$ and
denote $\tilde{I}^c = I^c\backslash\{d\}$ for $I\subseteq I_d$.
Then, applying it for fixed $X^{(d)}_{i_d}$ to the array of
functions $(\sum_{i_d}
h_\ii(x_1,\ldots,x_{d-1},X^{(d)}_{i_d})_{\ii_{I_{d-1}}}$, we get
by the Fubini theorem
\begin{align}
\E|\sum_{\ii} &h_\ii(\xdi)|^p \nonumber\\
&\le L_{d-1}^p\Big(\sum_{K\subseteq
I\subseteq I_{d-1}}\sum_{\mathcal{J}\in \mathcal{P}_{I\backslash
K}} p^{p(\#\tilde{I}^c +
\deg\mathcal{J}/2)}\sum_{\ii_{\tilde{I}^c}}\E_{I^c}\|(\sum_{i_d}h_\ii)_{\ii_I}\|_{K,\mathcal{J}}^p\Big)\nonumber,
\end{align}
where we have replaced the maxima in $\ii_{I^c}$ by sums (we can
afford this apparent loss, since we will be able to fix it with
Lemma \ref{sumy_na_maxima}). Now, from Lemma \ref{Talagrand}
(applied to $\E_d$) it follows that
\begin{align*}
\E_{d}\|(\sum_{i_d}h_\ii)_{\ii_I}\|_{K,\mathcal{J}}^p &\le
L^p\Big( (\E_d\|(\sum_{i_d}h_\ii)_{\ii_I}\|_{K,\mathcal{J}})^p +
p^{p/2}\|(h_\ii)_{\ii_{I\cup\{d\}}}\|_{K,\mathcal{J}\cup\{\{d\}\}}^p\\
&+p^p\sum_{i_d}\E_{d}\|(h_\ii)_{\ii_I}\|_{K,\mathcal{J}}^p\Big).
\end{align*}
Since $\tilde{I}^c = (I\cup\{d\})^c$,
$\deg{\mathcal{J}\cup\{\{d\}\}} = \deg\mathcal{J}+1$ and $\#I^c =
\#\tilde{I}^c + 1$, combining the above inequalities gives
\begin{align*}
\E|\sum_{\ii} h_\ii(\xdi)|^p &\le L_d^p\Big(\sum_{K\subseteq
I\subseteq I_d}\sum_{\mathcal{J}\in \mathcal{P}_{I\backslash K}}
p^{p(\#I^c +
\deg\mathcal{J}/2)}\E_{I^c}\sum_{\ii_{I^c}}\|(h_\ii)_{\ii_I}\|_{K,\mathcal{J}}^p\\
&+ \sum_{K\subseteq I\subseteq I_{d-1}}\sum_{\mathcal{J}\in
\mathcal{P}_{I\backslash K}} p^{p(\#\tilde{I}^c +
\deg\mathcal{J}/2)}\sum_{\ii_{\tilde{I}^c}}\E_{\tilde{I}^c}(\E_d\|(\sum_{i_d}h_\ii)_{\ii_I}\|_{K,\mathcal{J}})^p\Big).
\end{align*}
By applying Lemma \ref{crucial_lemma} to the second sum on the
right hand side, we get
\begin{align}\label{nierownosc_z_sumami}
\E|\sum_{\ii} h_\ii(\xdi)|^p &\le L_d^p\Big(\sum_{K\subseteq
I\subseteq I_d}\sum_{\mathcal{J}\in \mathcal{P}_{I\backslash K}}
p^{p(\#I^c +
\deg\mathcal{J}/2)}\E_{I^c}\sum_{\ii_{I^c}}\|(h_\ii)_{\ii_I}\|_{K,\mathcal{J}}^p\Big).
\end{align}

We can now finish the proof using Lemma \ref{sumy_na_maxima}. We
apply it to $\E_{I^c}$ for $I \neq I_d$, with $\#I^c$ instead of $d$ and $p/2$ instead of $p$
(for $p = 2$ the theorem is trivial, so we can assume that $p >
2$) and $\alpha = 2\#I^c + \deg\mathcal{J} + \#I^c$. Using the fact that $(p/2)^{\alpha\#I^c} \le L_d^{p}$ and
$\E\|(h_\ii)_{\ii_I}\|_{K,\mathcal{J}}^2 \le \sum_{\ii_I}\E_I|h_\ii|^2$, we get
\begin{align*}
\E_{I^c}&\sum_{\ii_{I^c}}\|(h_\ii)_{\ii_I}\|_{K,\mathcal{J}}^p \le
p^{-\alpha p/2}\tilde{L}_d^p\Big(p^{\alpha p/2}\E_{I^c}\max_{\ii_{I^c}}
\|(h_\ii)_{\ii_I}\|_{K,\mathcal{J}}^p \\
&+ \max_{J \subsetneq I^c}p^{\#J
p/2}\E_{J}\max_{\ii_{J}}(\sum_{\ii_{I^c\backslash J}}
\E_{I^c\backslash J}
\|(h_\ii)_{\ii_{I}}\|_{K,\mathcal{J}}^2\Big)^{p/2}\\
&\le \bar{L}_d^p\Big(\E_{I^c}\max_{\ii_{I^c}}
\|(h_\ii)_{\ii_I}\|_{K,\mathcal{J}}^p \\
&+p^{-(\#I^c + \deg\mathcal{J}/2)p}\max_{J\subseteq I^c}
\E_{J}\max_{\ii_{J}}(\sum_{\ii_{J^c}}
\E_{J^c} |h(\xdi)|^2)^{p/2}\Big) \\
&= \tilde{L}_d^p\Big(\E_{I^c}\max_{\ii_{I^c}}
\|(h_\ii)_{\ii_I}\|_{K,\mathcal{J}}^p +p^{-(\#I^c + \deg\mathcal{J}/2)p}\max_{J\subseteq
I^c}
\E_{J}\max_{\ii_{J}}\|(h_\ii)_{\ii_{J^c}}\|_{J^c,\emptyset}^p\Big),
\end{align*}
which allows us to replace the sums in $\ii_{I^c}$ on the right-hand
side of (\ref{nierownosc_z_sumami}) by the corresponding maxima,
proving the inequality in question.
\end{proof}

Theorem \ref{moment_estimates} gives a precise estimate for
moments of canonical Hilbert space valued $U$-statistics. In the
sequel however we will need a weaker estimate, using the
$\|\cdot\|_{K,\mathcal{J}}$ norms only for $I = I_d$ and
specialized to the case $h_\ii = h$. Before we formulate a proper corollary, let us
introduce
\begin{defi}
Let $h \colon \Sigma^d \to H$ be a canonical kernel. Let moreover $X_1,X_2, \ldots, X_d$ be i.i.d random variables with values
in $\Sigma$. Denote $X = (X_1,\ldots,X_d)$ and for $J \subseteq I_d$, $X_J = (X_j)_{j\in J}$.
For $K \subseteq I \subseteq I_d$ and $\mathcal{J} = \{J_1,\ldots,J_k\} \in \mathcal{P}_{I\backslash K}$,
we define
\begin{align*}
\|h\|_{K,\mathcal{J}} = \sup\big\{&\E_I \langle h(X), g(X_K)\rangle\prod_{j=1}^k f_j(X_{J_j})\colon
g\colon \Sigma^{\#K} \to H,\; \E |g(X_K)|^2 \le 1,\\
& f_j\colon \Sigma^{\#J_j} \to \R, \; \E f_j(X_{J_j}))^2 \le 1,\; j = 1,\ldots,k\big\}.
\end{align*}
In other words $\|h\|_{K,\mathcal{J}}$ is the $\|\cdot\|_{K,\mathcal{J}}$ of an array $(h_\ii)_{|\ii| = 1}$, with
$h_{(1,\ldots,1)} = h$.
\end{defi}

\paragraph{Remark} For $I = I_d$, $\|h\|_{K,\mathcal{J}}$ is a norm, whereas for $I \subsetneq I_d$, it is a random variable, depending on $X_{I^c}$.

It is also easy to see that if all the variables $X_i^{(j)}$ are i.i.d. and for all $|\ii|\le n$ we have $h_\ii = h$, then for any fixed value of
$\ii_{I^c}$,
\begin{displaymath}
\|(h_\ii)_{|\ii_I|\le n}\|_{K,\mathcal{J}} = \|h\|_{K,\mathcal{J}} n^{\#I/2},
\end{displaymath}
where $\|h\|_{K,\mathcal{J}}$ is defined with respect to any i.i.d. sequence $X_1,\ldots,X_d$ of the form
$X_j = X^{(j)}_{i_j}$ for $j \in I^c$.

We also have $\|h\|_{K,\mathcal{J}} \le \sqrt{\E_I |h(X)|^2}$, which together with the above observations allows us to derive the following
\begin{cor} \label{simplified_moments}For all $p \ge 2$, we have
\begin{align*}
\E|\sum_{\ii} h(\xdi)|^p \le& L_d^p\Big(\sum_{K\subseteq
I_d}\sum_{\mathcal{J}\in \mathcal{P}_{I_d\backslash
K}}p^{p\deg\mathcal{J}/2}n^{dp/2}\|h\|_{K,\mathcal{J}}^p \\
&+ \sum_{I\subsetneq
I_d}p^{p(d+\#I^c)/2}n^{\#Ip/2}\E_{I^c}\max_{\ii_{I^c}}(\E_{I}|h(\xdi)|^2)^{p/2}\Big)
\end{align*}
\end{cor}

The Chebyshev inequality gives the following  corollary for
bounded kernels
\begin{cor} If $h$ is bounded, then for all $t \ge 0$,
\begin{align*}
\p\Big(|\sum_{\ii}h(\xdi)|&\ge L_d (n^{d/2}(\E|h|^2)^{1/2} + t)\Big)
\le\\
&L_d\exp\big[-\frac{1}{L_d}\Big(\min_{ K\subsetneq
I_d,\mathcal{J}\in\mathcal{P}_{I_d\backslash K}}
\Big(\frac{t}{n^{d/2}\|h\|_{K,\mathcal{J}}}\Big)^{2/\deg(\mathcal{J})}
\Big)\wedge\\
&\quad\quad\quad\wedge\Big(\min_{I\subsetneq
I_d}\Big(\frac{t}{n^{\#I/2}\|(\E_I
|h|^2)^{1/2}\|_\infty}\Big)^{2/(d+\#I^c)}\Big)\Big].
\end{align*}
\end{cor}

Before we formulate the version of exponential inequalities that will be useful for the analysis of the LIL,
let us recall the classical definition of Hoeffding projections.

\begin{defi}
For an integrable kernel $h\colon \Sigma^d \to H$,
define $\pi_d h \colon \Sigma^k \to \R$ with the formula
\begin{displaymath}
\pi_d h(x_1,\ldots,x_k) = (\delta_{x_1} -
\mathbf{P})\times(\delta_{x_2}-\mathbf{P})\times\ldots\times(\delta_{x_d}-\mathbf{P})h,
\end{displaymath}
where $\mathbf{P}$ is the law of $X_1$.
\end{defi}

\paragraph{Remark} It is easy to see that $\pi_k h$ is canonical. Moreover $\pi_d h = h$ iff $h$ is canonical.

The following Lemma was proven for $H= \R$ in \cite{AmcLat} (Lemma 1). The proof given there works
for an arbitrary Banach space.
\begin{lemma}\label{moments_comparison}
Consider an arbitrary family of integrable kernels $h_{\ii} \colon
\Sigma^d \to H$, $|\ii| \le n$. For any
$p \ge 1$ we have
\begin{displaymath}
\big\|\sum_{|\ii|\le n}\pi_d h_\ii(\xdi)\big\|_p \le
2^d\big\|\sum_{|\ii|\le n}\edi h_\ii(\xdi)\big\|_p.
\end{displaymath}
\end{lemma}

In the sequel we will use exponential inequalities to
$U$-statistics generated by $\pi_d h$, where $h$ will be a
non-necessarily canonical kernel of order $d$. Since the kernel
$\tilde{h}((\varepsilon_1,X_1),\ldots,(\varepsilon_d,X_d)) =
\varepsilon_1\cdots\varepsilon_dh(X_1,\ldots,X_d)$, where
$\varepsilon_i$'s are i.i.d. Rademacher variables independent of
$X_i$'s is always canonical, Corollary \ref{simplified_moments},
Lemma \ref{moments_comparison} and the Chebyshev inequality give
us also the following corollary (note that
$\|\tilde{h}\|_{K,\mathcal{J}} = \|h\|_{K,\mathcal{J}}$)

\begin{cor} \label{tail_Hoeffding}If $h$ is bounded, then for all $p \ge 0$,
\begin{align*}
\p\Big(\Big|\sum_{\ii}\pi_d h(\xdi)\Big|&\ge L_d (n^{d/2}(\E|h|^2)^{1/2} +
t)\Big)
\le\\
&L_d\exp\Big[-\frac{1}{L_d}\Big(\min_{ K\subsetneq
I_d,\mathcal{J}\in\mathcal{P}_{I_d\backslash K}}
\Big(\frac{t}{n^{d/2}\|h\|_{K,\mathcal{J}}}\Big)^{2/\deg(\mathcal{J})}
\Big)\wedge\\
&\quad\quad\quad\wedge\Big(\min_{I\subsetneq
I_d}\Big(\frac{t}{n^{\#I/2}\|(\E_I
|h|^2)^{1/2}\|_\infty}\Big)^{2/(d+\#I^c)}\Big)\Big].
\end{align*}
\end{cor}

\section{The equivalence of several LIL statements\label{LIL_statements_equiv}}
In this section we will recall general results on the
correspondence of various statements of the LIL. We will state
them without proofs, since all of them have been proven in
\cite{GZh} and \cite{AmcLat} in the real case and the proofs can
be directly transferred to the Hilbert space case, with some
simple modifications that we will indicate.

Before we proceed, let us introduce the assumptions and notation common for the remaining part of the article.

\begin{itemize}
\item We assume that $(X_i)_{i\in \N}$, $(X_i^{(k)})_{i\in \N,
1 \le k \le d}$ are i.i.d. and $h\colon\Sigma^d \to H$ is a measurable function.

\item Recall that $(\varepsilon_i)_{i\in \N}$, $(\varepsilon_i^{(k)})_{i\in \N,
1 \le k \le d}$ are independent Rademacher variables, independent of $(X_i)_{i\in \N}$, $(X_i^{(k)})_{i\in \N,
1 \le k \le d}$.

\item To avoid technical problems with small values of $h$ let us also
define $\llog x = \loglog (x\vee e^e)$.

\item We will also occasionally write $X$ for $(X_1,\ldots, X_d)$ and
for $I\subseteq I_d$, $X_I = (X_i)_{i\in I}$. Sometimes we will
write simply $h$ instead of $h(X)$.

\item We will use the letter $K$ to denote constants depending only on the function $h$.
\end{itemize}

We will need the following simple fact
\begin{lemma}\label{nowy_lemat_nie_mam_sily_wymyslac_nazwy}
If $\E|h|^2/(\llog |h|)^d = K < \infty$ then $\E (|h|^2\wedge u)
\le L (\loglog u)^d$ with $L$ depending only on $K$ and $d$.
\end{lemma}

The next lemma comes from \cite{GZh}. It is proven there for $H = \R$ but the argument is valid also for general Banach spaces.
\begin{lemma}\label{GineZhang}
Let $h\colon \Sigma^d \to H$ be a symmetric function. There exist
constants $L_d$, such that if
\begin{equation}\label{lil}
\limsup_{n\to \infty} \frac{1}{(n\loglog n)^{d/2}}\big|\sum_{\ii
\in \nodiag{n}} h(\xui)\big| < C \; {\rm a.s.},
\end{equation}
then
\begin{equation}\label{first_series}
\sum_{n=1}^\infty \p\big(\big|\sum_{|\ii| \le 2^n} \edi
h(\xdi)\big| \ge D 2^{nd/2}\log^{d/2} n\big) < \infty
\end{equation}
for $D = L_d C$.
\end{lemma}

\begin{lemma}\label{lil_decoupledLIL}
For a symmetric function $h\colon \Sigma^d \to H$, the LIL
(\ref{lil}) is equivalent to the decoupled LIL
\begin{equation}\label{decoupled_lil}
\limsup_{n\to \infty}\frac{1}{(n\loglog n)^{d/2}}\big|\sum_{\ii
\in \nodiag{n}} h(\xdi)\big| < D\; {\rm a.s.},
\end{equation}
meaning that (\ref{lil}) implies (\ref{decoupled_lil}) with $D =
L_dC$, and conversely (\ref{decoupled_lil}) implies (\ref{lil})
with $C = L_d D$.
\end{lemma}

\begin{proof}
This is Lemma 8 in \cite{AmcLat}. The proof is the same as there,
one needs only to replace $l^\infty$ with $l^\infty(H)$ -- the
space of bounded $H$-valued sequences.
\end{proof}

The next lemma also comes from \cite{AmcLat} (Lemma 9). Although stated for
real kernels, its proof relies on an inductive argument with a
stronger, Banach-valued hypothesis.

\begin{lemma}\label{Montgomery_iterated}
There exists a universal constant $L < \infty$, such that for any
kernel $h\colon \Sigma^d \to H$ we have
\begin{displaymath}
\p\big(\max_{|\jj|\le n} \big|\sum_{\ii\colon i_k\le j_k ,
k=1\ldots d}h(\xdi)\big|\ge t\big) \le
L^d\p\big(\big|\sum_{|\ii|\le n} h(\xdi)\big| \ge t/L^d\big).
\end{displaymath}
\end{lemma}

\begin{cor}\label{from_series_to_lil}
Consider a kernel $h \colon \Sigma^d \to H$ and $\alpha > 0$. If
\begin{displaymath}
\sum_{n=1}^\infty\p(|\sum_{|\ii|\le 2^n} h(\xdi)| \ge
C2^{n\alpha}\log^\alpha n) < \infty,
\end{displaymath}
then
\begin{displaymath}
\limsup_{n\to \infty}\frac{1}{(n\loglog
n)^\alpha}\big|\sum_{|\ii|\le 2^n} h(\xdi) \big| \le L_{d,\alpha}C
\; {\rm a.s.}
\end{displaymath}
\end{cor}
\begin{proof}
Given Lemma \ref{Montgomery_iterated}, the proof is the same as
the one for real kernels, presented in \cite{AmcLat} (Corollary 1 therein).
\end{proof}

The next lemma shows that the contribution to a decoupled
U-statistic from the 'diagonal', i.e. from the sum over
multiindices $\ii\notin \nodiag{n}$ is negligible. The proof given
in \cite{AmcLat} (Lemma 10) is still valid, since the only part which cannot
be directly transferred to the Banach space setting is the
estimate of variance of canonical U-statistics, which is the same
in the real and general Hilbert space case.

\begin{lemma}\label{diagonal} If  $h\colon \Sigma^d \to H$ is canonical and satisfies
\begin{displaymath}
\E (|h|^2\wedge u) = \mathcal{O}((\loglog u)^{\beta}),
\end{displaymath}
for some $\beta$, then
\begin{equation}\label{diagonal_series}
\limsup_{n \to \infty}\frac{1}{(n\loglog
n)^{d/2}}\big|\sum_{\stackrel{|\ii| \le n}{\exists_{j\neq k} i_j =
i_k}} h(\xdi)\big| = 0\; {\rm a.s.}
\end{equation}
\end{lemma}

\begin{cor}\label{randomized_decoupled_series}
The randomized decoupled LIL
\begin{equation}\label{randomized_decoupled_lil}
\limsup_{n\to \infty} \frac{1}{(n\loglog n)^{d/2}}\big|\sum_{|\ii|
\le n} \edi h(\xdi)\big| < C
\end{equation}
is equivalent to (\ref{first_series}), meaning then if
(\ref{randomized_decoupled_lil}) holds then so does
(\ref{first_series}) with $D = L_d C$ and (\ref{first_series})
implies (\ref{randomized_decoupled_lil}) with $C = L_d D$.
\end{cor}

The proof is the same as for the real-valued case, given in \cite{AmcLat} (Corollary 2), one only needs to replace $h^2$ by $|h|^2$
and use the formula for the second moments in Hilbert spaces.

\begin{cor}\label{lil_dec_diagLIL}
For a symmetric, canonical kernel $h\colon \Sigma^d \to H$, the
LIL (\ref{lil}) is equivalent to the decoupled LIL 'with diagonal'
\begin{equation}\label{decoupled_diagonal_lil}
\limsup_{n \to \infty} \frac{1}{(n\loglog
n)^{d/2}}\big|\sum_{|\ii| \le n} h(\xdi)\big|< D
\end{equation}
again meaning that there are constants $L_d$ such that if
(\ref{lil}) holds for some $D$ then so does
(\ref{decoupled_diagonal_lil}) for $D = L_d C$, and conversely,
(\ref{decoupled_diagonal_lil}) implies (\ref{lil}) for $C = L_dD$.
\end{cor}
\begin{proof}
The proof is the same as in the real case (see \cite{AmcLat}, Corollary 3).
Although the integrability of the kernel guaranteed by the LIL is
worse in the Hilbert space case, it still allows one to use Lemma
\ref{diagonal}.
\end{proof}

\section{The canonical decoupled case}

Before we formulate the necessary and sufficient conditions for the bounded LIL in Hilbert spaces, we need
\begin{defi} For a canonical kernel $h\colon \Sigma^d \to H$, $K\subseteq I_d$, $\mathcal{J}=\{J_1,\ldots,J_k\}\in\mathcal{P}_{I_d\backslash K}$ and $u > 0$ we define
\begin{align*}
\|h\|_{K,\mathcal{J},u} = \sup\{&\E\langle
h(X),g(X_K)\rangle\prod_{i=1}^k f_i(X_{J_i})\colon g\colon
\Sigma^K \to H, \\
&f_i\colon\Sigma^{J_i}\to \R, \|g\|_2,\|f_i\|_2\le
1,\|g\|_\infty,\|f_i\|_\infty\le u\},
\end{align*}
where for $K= \emptyset$ by $g(X_K)$ we mean an element $g \in H$, and $\|g\|_2$ denotes just the norm of $g$ in $H$ (alternatively
we may think of $g$ as of a random variable measurable with respect to $\sigma((X_i)_{i\in\emptyset})$, hence constant). Thus the condition on $g$ becomes in this
case just $|g|\le 1$.
\end{defi}

\paragraph{Example} For $d=2$, the above definition reads as
\begin{align*}
\|h(X_1,X_2)\|_{\emptyset,\{\{1,2\}\},u} & = \sup\{|\E h(X_1,X_2)f(X_1,X_2)|\colon\\
&\quad\quad\quad\; \E f(X_1,X_2)^2\le 1, \|f\|_\infty\le u\},\\
\|h(X_1,X_2)\|_{\emptyset, \{\{1\}\{2\}\},u} & = \sup\{|\E h(X_1,X_2)f(X_1)g(X_2)|\colon\\
&\quad\quad\quad\;\E f(X_1)^2, \E g(X_2)^2\le 1\\
&\quad\quad\quad\; \|f\|_\infty, \|g\|_\infty \le u\},\\
\|h(X_1,X_2)\|_{\{1\},\{\{2\}\},u} & = \sup\{\E \langle f(X_1),h(X_1,X_2)\rangle g(X_2)\colon \\
&\quad\quad\quad\;\E |f(X_1)|^2,\E g(X_2)^2\le 1\\
&\quad\quad\quad\; \|f\|_\infty, \|g\|_\infty,  \le
u\}\\
\|h(X_1,X_2)\|_{\{1,2\},\emptyset,u} & = \sup\{\E \langle f(X_1,X_2),h(X_1,X_2)\rangle\colon\\
&\quad\quad\quad\; \E |f(X_1,X_2)|^2\le 1, \|f\|_\infty\le u\}.
\end{align*}
\begin{theorem}\label{decoupled_lil_theorem}
Let $h$ be a canonical $H$-valued symmetric kernel in $d$ variables. Then the
decoupled LIL
\begin{equation}\label{LIL_theorem}
\limsup_{n\to \infty} \frac{1}{n^{d/2}{(\loglog n)^{d/2}}}\big|
\sum_{|i| \le n}  h(\xdi)| < C
\end{equation}
holds if and only if
\begin{align}\label{nasc_integrability}
\E\frac{|h|^2}{(\llog|h|)^d} < \infty
\end{align}
and for all $K\subseteq I_d,\mathcal{J} \in
\mathcal{P}_{I_d\backslash K}$
\begin{equation}\label{nasc}
\limsup_{u\to \infty} \frac{1}{(\loglog
u)^{(d-\deg\mathcal{J})/2}}\|h\|_{K,\mathcal{J},u} < D.
\end{equation}
More precisely, if (\ref{LIL_theorem}) holds for some $C$ then (\ref{nasc})
is satisfied for $D =  L_dC$ and conversely, (\ref{nasc_integrability}) and (\ref{nasc}) implies
(\ref{LIL_theorem}) with $C = L_dD$.
\end{theorem}

\paragraph{Remark} Using Lemma \ref{nowy_lemat_nie_mam_sily_wymyslac_nazwy} one can easily check that the
condition (\ref{nasc}) with $D < \infty$ for $I=I_d$ is implied by (\ref{nasc_integrability}).

\section{Necessity}

The proof is a refinement of ideas from \cite{L2}, used to study random matrix approximations
of the operator norm of kernel integral operators.

\begin{lemma}
\label{PalZygmult}

If $a,t>0$ and $h$ is a nonnegative $d$-dimensional kernel  such that $N^{d}\E
h(X)\geq ta$ and $\|\E_{I}h(X)\|_{\infty}\leq N^{-\# I}a$ for all $\emptyset\subseteq I\subsetneq \{1,\ldots,d\}$,
then
\[
 \forall_{\lambda\in (0,1)}\
 \p(\sum_{|\ii|\leq N}h(\xdi)\geq \lambda ta)\geq
 (1-\lambda)^2\frac{t}{t+2^{d}-1}
 \geq (1-\lambda)^2 2^{-d}\min(1,t).
\]
\end{lemma}

\begin{proof}
We have
\begin{align*}
\E\Big(\sum_{|\ii|\le N} h(\xdi)\Big)^2 =& \sum_{|\ii|\le N}\sum_{|\jj|\le N} \E h(\xdi)h(\xd_\jj)\\
& = \sum_{I\subseteq I_d}\sum_{|\ii|\le N}\sum_{\stackrel{|\jj|\le N\colon}{\{k\colon i_k = j_k\} = I}} \E h(\xdi)h(\xd_\jj)\\
&= \sum_{I\subseteq I_d}\sum_{|\ii|\le N}\sum_{\stackrel{|\jj|\le N\colon}{\{k\colon i_k = j_k\} = I}} \E [h(\xdi)\E_{I^c}h(\xd_\jj)]\\
&\le N^{2d}(\E h(X))^2 + \sum_{\emptyset \neq I\subseteq I_d} N^{d+\#I^c} \E h(X) \|\E_{I^c} h(X)\|_\infty\\
&\le N^{2d}(\E h(X))^2 + (2^d-1)N^d a\E h(X)\\
&\le N^{2d}(\E h(X))^2 + (2^d-1)t^{-1}N^{2d}(\E h(X))^2\\
&\le \frac{t + 2^{d}-1}{t}\Big(\E \sum_{|\ii|\le n } h(\xdi)\Big)^2.
\end{align*}
The lemma follows now from the Paley-Zygmund inequality (see e.g. \cite{dlp2}, Corollary 3.3.2.), which says that for an arbitrary nonnegative random variable $S$,
\begin{displaymath}
\p(S \ge \lambda S) \ge (1-\lambda)^2\frac{(\E S)^2}{\E S^2}.
\end{displaymath}
\end{proof}

\begin{cor}

\label{smallsec}

Let $A\subseteq \Sigma^d$ be a measurable set, such that

\[
  \forall_{\emptyset\subsetneq I\subsetneq \{1,\ldots,d\}}
  \forall_{x_{I^c}\in \Sigma^{I^c}} \;\p_{I}((x_{I^c},X_{I})\in A)\leq N^{-\# I}.
\]

Then

\[
  \p(\exists_{|\ii|\leq N}\ \xdi\in A)\geq 2^{-d}
  \min(N^{d}\p(X\in A),1).
\]

\end{cor}

\begin{proof} We apply Lemma \ref{PalZygmult} with $h=I_{A}$,
$a=1$, $t=N^{d}\p(X\in A)$ and $\lambda\rightarrow 0+$.
\end{proof}

\begin{lemma}
\label{sumbelow}
Suppose that $Z_{j}$ are nonnegative r.v.'s, $p>0$ and $a_{j}\in
\R$ are such that
$\p(Z_{j}\geq a_j)\geq p$ for all $j$. Then

\[
\p(\sum_{j}Z_j\geq p\sum_{j}a_{j}/2)\geq p/2.
\]

\end{lemma}

\begin{proof} Let $\alpha:=\p(\sum_{j}Z_j\geq p\sum_{j}a_{j}/2)$,
then

\[
p\sum_j a_j\leq \E(\sum_j \min(Z_j,a_{j}))\leq
\alpha \sum_{j} a_{j} +p\sum_{j} a_{j}/2.
\]

\end{proof}

\begin{theorem}\label{series_integral}

Let $Y$ be a r.v. independent of $X_i^{(j)}$. Suppose that for each $n$, $a_{n}\in\R$,
$h_{n}$ is a $d+1$-dimensional nonnegative kernel such that

\[
\sum_{n}\p\Big(\sum_{|\ii|\leq 2^n}h_{n}(\xdi,Y)\geq
a_{n}\Big)<\infty.
\]

Let $p>0$, then there exists a constant $C_{d}(p)$ depending only on $p$ and $d$
such that the sets

\[
A_{n}:=\big\{x\in S^d\colon \forall_{n\leq m\leq 2^{d-1}n}\
\p_{Y}(h_{m}(x,Y)\geq C_d(p)2^{d(n-m)}a_{m})\geq p\big\},
\]
satisfy $\sum 2^{dn}\p(X\in A_n)<\infty$.
\end{theorem}

\begin{proof}

We will show by induction on $d$, that the assertion holds with
$C_{1}(p):=1$, $C_{2}(p):=12/p$ and

\[
  C_{d}(p):=12p^{-1}\max_{1\leq l\leq d-1}C_{d-l}(2^{-l-4}p/3)
  \mbox{ for } d=3,4,\ldots.
\]

For $d=1$ we have

\begin{align*}
\frac{1}{2}\min(2^n\p(X\in A_n),1)&\leq
\p(\exists_{|\ii|\leq 2^{n}}\ \xdi\in A_n)
\\
&=\p(\exists_{|\ii|\leq 2^{n}}\ \p_{Y}(h_n(\xdi,Y)\geq
a_n)\geq p)
\\
&\leq \p(\p_{Y}(\sum_{|\ii|\leq 2^{n}}h_n(\xdi,Y)\geq
a_n)\geq p)
\\
&\leq p^{-1}\p(\sum_{|\ii|\leq 2^{n}}h_n(\xdi,Y)\geq
a_n).
\end{align*}

Before investigating the case $d>1$ let us define

\[
  \tilde{A}_n:=A_{n}\setminus\bigcup_{m>n}A_m.
\]

The sets $\tilde{A}_n$ are pairwise disjoint and obviously $\tilde{A}_n\subset A_n$.
Notice that since $C_{d}(p)\geq 1$,

\[
  \p(X\in A_n)\leq \p(\p_{Y}(h_n(X,Y)\geq a_n))
  \leq p^{-1}\p(\sum_{|\ii|\leq 2^n}h_{n}(\xdi,Y)\geq a_n).
\]
Hence $\sum_{n}\p(X\in A_n)<\infty$, so $\p(X\in
\limsup A_n)=0$.
But if $x\notin \limsup A_n$, then
$\sum_{n}2^{nd}I_{A_{n}}(x)\leq \sum_{n}2^{nd+1}I_{\tilde{A}_n}(x)$. So it is enough to show
that $\sum 2^{dn}\p(X\in \tilde{A}_n)<\infty$.

\paragraph{Induction step} Suppose that the statement holds for all $d'<d$, we will
show it for $d$. First we will inductively construct sets

\[
\tilde{A}_{n}=A_{n}^{0}\supset A_{n}^1\supset\ldots \supset A_{n}^{d-1}
\]
such that for $1\leq l\leq d-1$,

\begin{equation}
\label{cond1}
\forall_{\emptyset\subsetneq I\subsetneq \{1,\ldots,d-1\},\;\#I\leq l}\;
\forall_{x_{I^c}}\ \p_{I}((x_{I^c},X_{I})\in A_n^{l})\leq
2^{-n\#I}
\end{equation}
and

\begin{equation}
\label{cond2}
 \sum_{n}2^{nd}\p(X\in A_{n}^{l-1}\setminus A_{n}^l)<\infty.
\end{equation}

Suppose that $1\leq l\leq d-1$ and the set $A_n^{l-1}$ was already defined.
Let $I\subset\{1,\ldots,d\}$ be such that $\#I=l$ and let $j\in I$. Notice
that

\[
\p_I((x_{I^c},X_{I})\in
A_n^{l-1})=\E_{j}\p_{I\setminus\{j\}}
((x_{I^c},X_j,X_{I\setminus\{j\}})\in A_n^{l-1})
\leq 2^{-n(l-1)}
\]
by the property (\ref{cond1}) of the set $A_n^{l-1}$. Let us define for
$n(l-1)+1\leq k\leq nl$,

\[
B_{n,k}^{I}:=\{x_{I^c}\colon \p_I((x_{I^c},X_{I})\in
A_n^{l-1})\in
(2^{-k},2^{-k+1}]\}
\]
and

\[
B_{n}^{I}:=\bigcup_{k=n(l-1)+1}^{nl}B_{n,k}^{I}=\{x_{I^c}\colon
\p_I((x_{I^c},X_{I})\in A_n^{l-1})>2^{-nl}\}.
\]

We have

\begin{align*}
  \sum_{n}2^{dn}\p(X\in A_{n}^{l-1}, X_{I^{c}}\in B_{n}^{I})
  &\leq
  2\sum_{n}\sum_{k=n(l-1)+1}^{nl}2^{dn-k}\p(X_{I^{c}}\in B_{n}^{I})
\\
  &=2\E k_{1}^{I}(X_{I^c}),
\end{align*}
where

\[
  k_1^{I}(x_{I^c}):=\sum_{n}\sum_{k=n(l-1)+1}^{nl}2^{dn-k}I_{B_{n,k}^I}(x_{I^c}).
\]
Let $m\ge 1$ and

\[
  C_{m}^{I}:=\{x_{I^c}\colon 2^{(m+1)(d-l)}> k_1(x_{I^c})\geq 2^{m(d-l)}\}.
\]
Notice that for $n>m$ and $k\leq nl$, $2^{dn-k}\geq 2^{(d-l)(m+1)}$, moreover

\[
  \sum_{n<m/2}\sum_{k=n(l-1)+1}^{nl}2^{dn-k}\leq \sum_{n<m/2}2^{(d-l+1)n}\leq
  \frac{4}{3}2^{(d-l+1)(m-1)/2}\leq \frac{2}{3}2^{(d-l)m}.
\]
Hence

\begin{align}\label{oszacowanie}
  x_{I^c}\in C_{m}^{I}\ \Rightarrow\
  \sum_{m/2\leq n\leq m}\sum_{k=n(l-1)+1}^{nl}2^{dn-k}I_{B_{n,k}^I}(x_{I^c})\geq
  \frac{1}{3}2^{(d-l)m}.
\end{align}
Let $m\leq r\leq 2^{d-2}m$,
if $m/2\leq n\leq m$, then since $A_{n}^{l-1}\subset A_n$ we have for all $x\in S^d$,

\[
  \p_{Y}(h_{r}(x,Y)\geq C_{d}(p)2^{d(n-r)}a_r I_{A_n^{l-1}}(x))\geq p,
\]
therefore, since $A_{n}^{l-1}\subset\tilde{A}_n$ are pairwise disjoint,

\[
  \p_{Y}\Big(h_{r}(x,Y)\geq C_{d}(p)2^{-dr}a_r
  \sum_{m/2\leq n\leq m}2^{dn}I_{A_n^{l-1}}(x)\Big)\geq p.
\]
Hence, by Lemma \ref{sumbelow},

\begin{equation}
\label{cased_est1}
\p_{Y}\Big(\sum_{|\ii_I|\leq
2^r}h_{r}(x_{I^c},\xdii,Y)\geq
\frac{p}{2}C_{d}(p)2^{-dr}a_r
  \sum_{|\ii_I|\leq 2^r}k_{2,x_{I^{c}}}(\xdii)\Big)\geq \frac{p}{2},
\end{equation}
where

\[
  k_{2,x_{I^{c}}}(x_I):=\sum_{m/2\leq n\leq m}2^{dn}I_{A_n^{l-1}}(x_{I^c},x_I).
\]

We have $\|k_{2,x_{I^{c}}}\|_{\infty}\leq 2^{dm}$ and
for $\emptyset\neq J\subsetneq I$, by the property (\ref{cond1}) of $A_n^{l-1}$,

\begin{align*}
\E_{J}k_{2,x_{I^c}}(x_{I\setminus J},X_J)&
=\sum_{m/2\leq n\leq m}2^{dn}
\p_J\Big((x_{I^c},x_{I\setminus J},X_J)\in A_{n}^{l-1}\Big)
\\
&\leq \sum_{m/2\leq n\leq m}2^{(d-\#J)n}\leq 2^{(d-\#J)m+1}.
\end{align*}
Moreover for $x_{I^c}\in C_m^I$, by the definition of $B_{n,k}^I$ and (\ref{oszacowanie}),

\begin{align*}
\E k_{2,x_{I^c}}(X_{I})&\geq \sum_{m/2\leq n\leq m}
\sum_{k=n(l-1)+1}^{nl}2^{dn}
\p_{I}((x_{I^c},X_I)\in A_n^{l-1})I_{B_{n,k}^I}(x_{I^c})
\\
&\geq \sum_{m/2\leq n\leq m}\sum_{k=n(l-1)+1}^{nl}
2^{dn-k}I_{B_{n,k}^I}(x_{I^c})\geq
\frac{1}{3}2^{(d-l)m}.
\end{align*}
Therefore by Lemma \ref{PalZygmult} (with $l$ instead of $d$ and
$a=2^{(d-l)m+rl+1},t=1/6,N=2^{r}, \lambda=1/2$), for $m \le r \le 2^{d-2}m$,

\[
  \p\Big(\sum_{|\ii_{I}|\leq 2^r}k_{2,x_{I^c}}(\xdii)\geq \frac{1}{6}2^{(d-l)m+rl}\Big)\geq \frac{1}{3}2^{-l-3}.
\]
Combining the above estimate with (\ref{cased_est1}) we get (for $x_{I^c}\in C_m^I$
and $m\leq r\leq 2^{d-2}m$),

\[
  \p_{I,Y}\Big(\sum_{|\ii_{I}|\leq 2^r}h_{r}(x_{I^c},\xdii,Y)
  \geq \frac{p}{12}C_{d}(p)2^{(d-l)(m-r)}a_r\Big)
  \geq \frac{1}{3}2^{-l-4}p.
\]
Let us define $\tilde{Y}:=((X_{i}^{(j)})_{j\in I},Y)$ and
$\tilde{h}_{n}(x_{I^c},\tilde{Y}):=\sum_{|\ii_{I}|\leq 2^n}
h_n(x_{I^c},\xdii,Y)$.
Then
\[
  \sum_{n}\p(\sum_{|\ii_{I^c}|\leq 2^n}\tilde{h}_{n}(\xdiic,\tilde{Y})
  \geq a_n)
  =\sum_{n}\p(\sum_{|\ii|\leq 2^n}h_{n}(\xdi,Y)\geq a_n)
  <\infty.
\]
Moreover (since $C_{d}(p)\geq 12p^{-1}C_{d-l}(2^{-l-4}p/3)$),

\[
  C_{m}^{I}\subseteq\Big\{\forall_{m\leq r\leq 2^{d-2}m}\
  \p_{\tilde{Y}}(\tilde{h}_r(x_{I^c},\tilde{Y})\geq C_{d-l}(2^{-l-4}p/3)2^{(d-l)(m-r)}a_r)
  \geq 2^{-l-4}p/3\Big\}.
\]
Hence by the induction assumption,

\begin{displaymath}
\sum_{m}2^{(d-l)m}\p(X_{I^c}\in C_m^{I})<\infty,
\end{displaymath}
so
$\E k_{1}^{I}(X_{I^{c}})<\infty$ and thus

\begin{equation}
\label{induct1}
 \forall_{\# I=l}\
 \sum_{n}2^{dn}\p(X\in A_{n}^{l-1}, X_{I^{c}}\in B_{n}^{I})<\infty.
\end{equation}
We set

\[
  A_{n}^{l}:=\{x\in A_{n}^{l-1}\colon x_{I^{c}}\notin B_{n}^{I} \mbox{ for all }
  I\subset\{1,\ldots,d\}, \#I=l\}.
\]

The set $A_{n}^{l}$ satisfies the condition (\ref{cond1}) by the definition
of $B_{n}^{I}$ and the property (\ref{cond1}) for $A_{n}^{l-1}$. The
condition (\ref{cond2}) follows by (\ref{induct1}).

Notice that the set $A_{n}^{d-1}$ satisfies the assumptions of Corollary
\ref{smallsec} with $N=2^n$, therefore if $C_{d}(p)\geq 1$,

\begin{align*}
2^{-d}\min(1,2^{nd}\p(X\in A_{n}^{d-1}))
&\leq \p(\exists_{|\ii|\leq 2^{n}}\ \xdi\in A_{n}^{d-1})
\leq \p(\exists_{|\ii|\leq 2^{n}}\ \xdi\in
\tilde{A}_{n})
\\
&\leq \p(\exists_{|\ii|\leq 2^{n}}\
\p_{Y}(h_{n}(\xdi,Y)\geq C_d(p)a_n)\geq p)
\\
&\leq \p(
\p_{Y}(\sum_{|\ii|\leq 2^n}h_{n}(\xdi,Y)\geq a_n)\geq p)
\\
&\leq p^{-1}\p(\sum_{|\ii|\leq 2^n}h_{n}(\xdi,Y)\geq
a_n).
\end{align*}
Therefore $\sum_{n}2^{nd}\p(X\in A_{n}^{d-1})<\infty$, so
by (\ref{cond2}) we get

\[
\sum_n 2^{nd}\p\Big(X\in \tilde{A}_{n})=
\sum_{n} 2^{nd}(\sum_{l=1}^{d-1}\p(X\in
A_{n}^{l-1}\setminus A_{n}^l)
+\p(X\in A_{n}^{d-1})\Big)<\infty.
\]
\end{proof}

\begin{cor}
\label{convergent_series}
If

\[
  \sum_{n}\p\Big(\sum_{|\ii|\leq 2^n}h^2(\xdi)\geq
  \varepsilon 2^{nd}(\log n)^{\alpha}\Big)<\infty
\]

for some $\varepsilon>0$ and $\alpha\in \R$,

then $\E \frac{h^2}{(\llog |h|)^{\alpha}}<\infty$.

\end{cor}
\begin{proof}
We apply Theorem \ref{series_integral} with $h_n = h^2$ and $a_n = \varepsilon2^{nd}\log^d n$ in the degenerate case when $Y$ is deterministic. It is easy to notice that
$h^2 \ge \tilde{C}_d(p,\varepsilon)2^{dn}\log^d n$ implies that
\begin{displaymath}
\forall_{n\le m\le 2^{d-1}n} \;h^2 \ge C_d(p)2^{d(n-m)}a_m.
\end{displaymath}
\end{proof}

To prove the necessity part of Theorem \ref{decoupled_lil_theorem} we will also need the following Lemmas
\begin{lemma}[\cite{AmcLat}, Lemma 12] \label{variance}
Let $g \colon \Sigma^d \to \R$ be a square integrable function.
Then
\begin{displaymath}
\Var (\sum_{|\ii| \le n} g(\xdi)) \le (2^d - 1) n^{2d-1}\E g(X)^2.
\end{displaymath}
\end{lemma}
\begin{lemma}[\cite{AmcLat}, Lemma 5] \label{calkowalnosc} If $\E (|h|^2\wedge u) = \mathcal{O}((\loglog
u)^{\beta})$ then
\begin{displaymath}
\E |h|\ind{|h| \ge s} = \mathcal{O}(\frac{(\loglog
s)^{\beta}}{s}).
\end{displaymath}
\end{lemma}

\begin{lemma}\label{chaos}
Let $(a_\ii)_{\ii \in I_n^d}$ be a $d$--indexed array of vectors from a Hilbert space $H$. Consider a random variable
\begin{displaymath}
S := \Big|\sum_{|\ii| \le n}
a_\ii\prod_{k=1}^d\varepsilon_{i_k}^{(k)}\Big| = \Big|\sum_{|\ii|\le
n}a_\ii\edi\Big|.
\end{displaymath}
For any set $K\subseteq I_d$ and a partition $\mathcal{J} = \{J_1,\ldots,J_m\} \in
\mathcal{P}_{I_d\backslash K}$ let us define
\begin{align*}
\|(a_\ii)\|_{K,\mathcal{J},p}^\ast := \sup\Big\{&|\sum_{|\ii| \le n}
\langle a_\ii,\alpha^{(0)}_{\ii_K}\rangle\prod_{k=1}^m \alpha_{\ii_{J_k}}^{(k)}| \colon
\sum_{\ii_K}|\alpha^{(0)}_{\ii_K}|^2 \le 1,
\sum_{\ii_{J_k}}(\alpha_{\ii_{J_k}}^{(k)})^2 \le p,\\
&\forall_{i_{\max J_k}\in I_n}
\sum_{\ii_{\nomax{J_k}}}(\alpha_{\ii_{J_k}}^{(k)})^2 \le 1, \; k
= 1,\ldots,m\Big\},
\end{align*}
where $\nomax{J} = J\backslash \{\max J\}$ (here  $\sum_{\ii_{\emptyset}} a_\ii = a_\ii$).

Then, for all $p \ge 1$,
\begin{displaymath}
\|S\|_p \ge \frac{1}{L_d} \sum_{K\subseteq I_d, \mathcal{J}\in \mathcal{P}_{I_d\backslash K}}
\|(a_\ii)\|_{K,\mathcal{J},p}^\ast.
\end{displaymath}
In particular for some constant $c_d$
\begin{displaymath}
\p(S \ge c_d\sum_{K\subseteq I_d, \mathcal{J}\in \mathcal{P}_{I_d\backslash K}}
\|(a_\ii)\|_{K,\mathcal{J},p}^\ast) \ge c_d\wedge e^{-p}.
\end{displaymath}
\end{lemma}

\paragraph{Remark} For $K=\emptyset$, we define
\begin{align*}
\|(a_\ii)\|_{\emptyset,\mathcal{J},p}^\ast := \sup\Big\{&\Big|\sum_{|\ii| \le n}
a_\ii\prod_{k=1}^m \alpha_{\ii_{J_k}}^{(k)}\Big| \colon (\alpha_{\ii_{J_k}}^{(k)})_{\ii_{J_k}} \in \R^{(I_n^{\#J_k})},
\sum_{\ii_{J_k}}(\alpha_{\ii_{J_k}}^{(k)})^2 \le p,\\
&\forall_{i_{\max J_k}\in I_n}
\sum_{\ii_{\nomax{J_k}}}(\alpha_{\ii_{J_k}}^{(k)})^2 \le 1, \; k
= 1,\ldots,m\Big\},
\end{align*}

It is also easy to see that for a $d$-indexed matrix, $\|(a_\ii)_\ii\|_{I_d,\{\emptyset\},p} = \sqrt{\sum_{\ii} |a_\ii|^2} = \|S\|_2$ and thus does
not depend on $p$. Since it will not be important in the applications, we keep a uniform notation with the subscript $p$.

\paragraph{Examples}
For $d=1$, we have
\begin{align*}
\|(a_i)_{i\le n}\|^\ast_{\emptyset,\{\{1\}\},p}&=\sup\big\{\big|\sum_{i=1}^n
a_i\alpha_i\big|\colon \sum_{i=1}^n \alpha_i^2\le p, |\alpha_i|\le
1,\; i=1,\ldots,n\big\},\\
\|(a_i)_{i\le n}\|^\ast_{\{1\},\emptyset,p} &= \sup\Big\{\sum \langle a_i, \alpha_i\rangle \colon \sum|\alpha_i|^2\le 1\Big\} = \sqrt{\sum_{i=1}^n|a_i|^2},
\end{align*}
whereas for $d=2$, we get
\begin{align*}
\|(a_{ij})_{i,j\le n}\|^\ast_{\emptyset,\{\{1\},\{2\}\},p} & =
\sup\big\{\big|\sum_{i,j=1}^n a_{ij}\alpha_i\beta_j\big|\colon
\sum_{i=1}^n \alpha_i^2 \le p,\sum_{j=1}^n\beta_j^2\le p,\\
&\phantom{xxxxxx} \forall_{i\in I_n}|\alpha_i|\le 1,\forall_{j\in
I_n}|\beta_j|\le 1\big\},
\\
\|(a_{ij})_{i,j\le n}\|^\ast_{\emptyset,\{I_2\},p} & =
\sup\big\{\big|\sum_{i,j=1}^n a_{ij}\alpha_{ij}\big|\colon
\sum_{i,j=1}^n \alpha_{ij}^2 \le p, \forall_{j\in I_n}\sum_{i=1}^n
\alpha_{ij}^2\le 1\big\},\\
\|(a_{ij})_{i,j\le n}\|^\ast_{\{1\},\{\{2\}\},p} & =
\sup\big\{\big|\sum_{i,j=1}^n \langle a_{ij},\alpha_i\rangle\beta_j\big|\colon
\sum_{i=1}^n |\alpha_i|^2 \le 1,\\
&\phantom{xxxxxx}\sum_{j=1}^n\beta_j^2\le p, \forall_{j\in
I_n}|\beta_j|\le 1\big\},\\
\|(a_{ij})_{i,j\le n}\|^\ast_{I_2,\emptyset,p} & =
\sup\big\{\big|\sum_{i,j=1}^n \langle a_{ij},\alpha_{ij}\rangle \big|\colon
\sum_{i,j=1}^n \alpha_{ij}^2 \le 1 \big\} = \sqrt{\sum_{ij} |a_{ij}|^2}.
\end{align*}
\begin{proof}[Proof of Lemma \ref{chaos}] We will combine the classical hypercontractivity property of Rademacher chaoses (see e.g. \cite{dlp2}, p. 110-116) with Lemma 3 in \cite{AmcLat},
which says that for $H = \R$ we have
\begin{align}\label{real_chaos}
\|S\|_p \ge \frac{1}{L_d}\sum_{\mathcal{J}\in\mathcal{P}_{I_d}} \|(a_\ii)\|_{\emptyset,\mathcal{J},p}.
\end{align}
Since $\|(a_\ii)\|_{I_d,\{\emptyset\},p} = \sqrt{\sum_{\ii}|a_\ii|^2} = \|S\|_2$, the inequality
$\|S\|_p \ge L^{-1}\|(a_\ii)\|_{I_d,\{\emptyset\},p}$ is just Jensen's inequality ($p \ge 2$) or the aforesaid
hypercontractivity of Rademacher chaos ($p \in (1,2)$). On the other hand, for $K\neq I_d$ and $\mathcal{J} \in \mathcal{P}_{I_d\backslash K}$, we have
\begin{align*}
\|S\|_p &= \Big(\E_{I_d\backslash K} \E_K \Big|\sum_{\ii_K}\prod_{k\in K}\varepsilon_{i_k}^{(k)}\sum_{\ii_{I_d\backslash K}} a_\ii \prod_{k\notin K} \varepsilon_{i_k}^{(k)}\Big|^p\Big)^{1/p}\\
&\ge \frac{1}{L_{\#K}}\Big(\E_{I_d\backslash K} \Big(\sum_{\ii_K} \Big|\sum_{\ii_{I_d\backslash K}} a_\ii \prod_{k\notin K} \varepsilon_{i_k}^{(k)}\Big|^2\Big)^{p/2}\Big)^{1/p}\\
& = \frac{1}{L_{\#K}}\Big(\E_{I_d\backslash K} \sup_{\sum_{\ii_K}|\alpha_{\ii_K}^{(0)}|^2\le 1}\Big|\sum_{\ii_{I_d\backslash K}} \sum_{\ii_K}\langle \alpha_{\ii_K}^{(0)},
a_\ii\rangle\prod_{k\notin K} \varepsilon_{i_k}^{(k)}\Big|^{p}\Big)^{1/p}\\
&\ge \frac{1}{L_{\#K}}\Big(\sup_{\sum_{\ii_K}|\alpha_{\ii_K}^{(0)}|^2\le 1}\E_{I_d\backslash K} \Big|\sum_{\ii_{I_d\backslash K}} \sum_{\ii_K}\langle \alpha_{\ii_K}^{(0)},
a_\ii\rangle\prod_{k\notin K} \varepsilon_{i_k}^{(k)}\Big|^{p}\Big)^{1/p}\\
&\ge \frac{1}{L_{\#K}L_{d-\#K}}\sup_{\sum_{\ii_K}|\alpha_{\ii_K}^{(0)}|^2\le 1}\Big\|(\sum_{\ii_K}\langle \alpha^{(0)}_{\ii_K},a_\ii\rangle)_{\ii_{I_d\backslash K}}\Big\|_{\emptyset,\mathcal{J},p}\\
& = \frac{1}{L_{\#K}L_{d-\#K}}\|(a_\ii)\|_{K,\mathcal{J},p},
\end{align*}
where the first inequality follows from hypercontractivity applied conditionally on $(\varepsilon^{(k)}_i)_{k\notin K,i\in I_n}$, the second is Jensen's inequality and the
third is (\ref{real_chaos}) applied for a chaos of order $d-\#K$.

The tail estimate follows from moment estimates by the Paley-Zygmund inequality and the inequality
$\|(a_\ii)\|_{K,\mathcal{J},tp} \le t^{\deg\mathcal{J}} \|(a_\ii)\|_{K,\mathcal{J},p}$ for $t \ge 1$ just like in \cite{KG, L3}.
\end{proof}

\begin{proof}[Proof of necessity]

First we will prove the integrability condition (\ref{nasc_integrability}).
Let us notice that by classical hypercontractive estimates for Rademacher chaoses and the Paley-Zygmund inequality (or by Lemma \ref{chaos}), we have
\begin{displaymath}
\p_\varepsilon\Big(\Big|\sum_{|\ii|\le 2^n} \edi h(\xdi)\Big| \ge c_d \sqrt{\sum_{|\ii|\le 2^n} h(\xdi)^2}\Big) \ge c_d
\end{displaymath}
for some constant $c_d > 0$. By the Fubini theorem it gives
\begin{displaymath}
\p_\varepsilon\Big(\Big|\sum_{|\ii|\le 2^n} \edi h(\xdi)\Big| \ge D 2^{nd/2}\log^{d/2}n\Big) \ge c_d\p\Big(\sum_{|\ii|\le 2^n} h(\xdi)^2 \ge D^2 c_d^{-2}2^{nd}\log^d n\Big),
\end{displaymath}
which together with Lemma \ref{GineZhang} yields
\begin{displaymath}
\sum_n \p\Big(\sum_{|\ii|\le 2^n} h(\xdi)^2 \ge D^2 c_d^{-2}2^{nd}\log^d n\Big) < \infty.
\end{displaymath}
The integrability condition (\ref{nasc_integrability}) follows now from Corollary \ref{convergent_series}.

Before we proceed to the proof of (\ref{nasc}), let us notice that (\ref{nasc_integrability}) and Lemma \ref{nowy_lemat_nie_mam_sily_wymyslac_nazwy} imply that
\begin{align}\label{growth_condition}
\E (|h|^2\wedge u) \le K(\loglog u)^d
\end{align}
 for $n$ large enough. The proof of (\ref{nasc}) can be now obtained by adapting
the argument for the real valued case.

Since $\lim_{n \to \infty}
\sum_{k=n}^{2n}\frac{1}{k} = \log 2$, (\ref{first_series}) implies that there exists $N_0$, such
that for all $N > N_0$, there exists $N \le n \le 2N$, satisfying
\begin{equation}\label{co_drugie_n}
\p\Big(\Big|\sum_{|\ii|\le 2^n}\edi h(\xdi)\Big| > L_dC2^{nd/2}\log^{d/2} n\Big) <
\frac{1}{10n}.
\end{equation}

Let us thus fix $N > N_0$ and consider $n$ as above. Let
$K \subseteq I_d$, $\mathcal{J} = \{J_1,\ldots,J_k\} \in \mathcal{P}_{I_d\backslash K}$. Let us
also fix functions $g\colon \Sigma^{\#K} \to H$, $f_j \colon \Sigma^{\#J_j} \to \R$,
$j=1,\ldots,k$, such that
\begin{align*}
\|g(X_k)\|_2 &\le 1, \|g(X_K)\|_\infty \le 2^{n/(2k+3)},\\
\|f_j(X_{J_j})\|_2 &\le 1, \|f_j(X_{J_j})\|_\infty \le 2^{n/(2k+3)}.
\end{align*}

The Chebyshev inequality gives
\begin{equation}\label{Chebyshev}
\p(\sum_{|\ii_{J_j}|\le 2^n} f_j(\xdij)^2\log n \le 10\cdot 2^{d}
2^{\#J_j n}\log n) \ge 1 - \frac{1}{10\cdot 2^d}.
\end{equation}
Similarly, if $K \neq \emptyset$,
\begin{equation}\label{Chebyshev1}
\p(\sum_{|\ii_{K}|\le 2^n} |g(\xdik)|^2\le 10\cdot 2^{d}
2^{\#K n}) \ge 1 - \frac{1}{10\cdot 2^d}
\end{equation}
and for $K=\emptyset$, $|g|\le 1$ (recall that
for $K = \emptyset$, the function $g$ is constant).

Moreover for $j = 1,\ldots,k$ and sufficiently large $N$,
\begin{align*}
\sum_{|\ii_{\nomax{J_j}}| \le 2^n} \frac{1}{2^{n\#J_j}}
f_j(\xdij)^2 \cdot \log n &\le \frac{2^{n\#
\nomax{J_j}}2^{2n/(2k+3)} \log n}{2^{n \# J_j}}
\\
&\le \frac{2^{2n/(2k+3)}\log n}{2^n} \le 1.
\end{align*}

Without loss of generality we may assume that the sequences
$(X_i^{(j)})_{i,j}$ and $(\varepsilon_i^{(j)})_{i,j}$ are defined
as coordinates of a product probability space. If for each
$j=1,\ldots,k$ we denote the set from (\ref{Chebyshev}) by $A_k$, and the set from (\ref{Chebyshev1}) by $A_0$,
we have $\p(\bigcap_{j=0}^k A_k) \ge 0.9$. Recall now Lemma
\ref{chaos}. On $\bigcap_{j=0}^k A_k$ we can estimate the
$\|\cdot\|^\ast_{K,\mathcal{J},\log n}$ norms of the matrix
$(h(\xdi))_{|\ii|\le 2^n}$ by using the test sequences
\begin{displaymath}
\alpha_{\ii_{J_j}}^{(j)} = \frac{f_j(\xdij)\sqrt{\log
n}}{10^{1/2}2^{d/2}2^{n\#J_j/2}}
\end{displaymath}
for $j = 1,\ldots, k$ and
\begin{displaymath}
\alpha_{\ii_{K}}^{(0)} = \frac{g(\xdik)}{10^{1/2}2^{d/2}2^{n\#K/2}}.
\end{displaymath}
Therefore with probability at least 0.9 we have

\begin{align}\label{oszacowanie_normy}
&\|(h(\xdi))_{|\ii|\le 2^n}\|^\ast_{K,\mathcal{J},\log n} \\
&\ge
\frac{(\log n)^{k/2}}{2^{d(k+1)/2}10^{(k+1)/2}2^{(\#K+\sum_j\# J_j)n/2}}
|\sum_{|\ii|\le 2^n}\langle g(\xdik), h(\xdi)\rangle\prod_{j=1}^k f_j(\xdij)| \nonumber\\
&= \frac{(\log n)^{k/2}}{2^{d(k+1)/2}10^{(k+1)/2}2^{dn/2}}|\sum_{|\ii|\le
2^n} \langle g(\xdik), h(\xdi)\rangle\prod_{j=1}^k f_j(\xdij)|.\nonumber
\end{align}

Our aim is now to further bound from below the right hand side of
the above inequality, to have, via Lemma \ref{chaos}, control from
below on the conditional tail probability of $\sum_{|\ii|\le
2^n}\edi h(\xdi)$, given the sample $(X_{i}^{(j)})$.

From now on let us assume that
\begin{equation}
|\E \langle g(X_K),h(X)\rangle\prod_{j=1}^kf_j(X_{J_j})| > 1. \label{assumption}
\end{equation}
The Markov inequality, (\ref{growth_condition}) and Lemma \ref{calkowalnosc} give
\begin{align}\label{Markow}
\p\big(|&\sum_{|\ii|\le 2^n} \langle g(\xdk), h(\xdik)\rangle\ind{|h(\xdi)| >
2^n}\prod_{j=1}^kf_j(\xdij)| \ge \frac{2^{nd}|\E
\langle g,h\rangle\prod_{j=1}^kf_j|}{4}\big) \nonumber\\
&\le 4\frac{2^{nd}(\|g\|_\infty\prod_{j=1}^{k}\|f_j\|_\infty)\cdot
\E|h|\ind{|h|> 2^n}}{2^{nd}|\E\langle g, h\rangle\prod_{j=1}^kf_j|} \le
42^{n(k+1)/(2k+3)}\E|h|\ind{|h|> 2^n} \nonumber\\
&\le 4K \frac{(\log n)^{d}}{2^{\frac{n(k+2)}{2k+3}}}.
\end{align}
Let now $h_n = h\ind{|h|\le 2^n}$. By the Chebyshev inequality,
Lemma \ref{variance} and (\ref{growth_condition})
\begin{align}\label{czebyszew}
\p\Bigg(|\sum_{|\ii|\le 2^n} \langle g(\xdik), &h_n(\xdi)\rangle\prod_{j=1}^kf_j(\xdij) -
2^{nd}\E \langle g,h_n\rangle
\prod_{j=1}^kf_j| \nonumber\\
&\ge \frac{2^{nd}}{5}|\E \langle g,h_n\rangle\prod_{j=1}^kf_j|\Bigg) \nonumber\\
&\le 25\frac{\Var(\sum_{|\ii|\le 2^n}\langle g(\xdik),h_n(\xdi)\rangle\prod_{j=1}^kf_j(\xdij))}{2^{2nd}|\E \langle g,h_n\rangle\prod_{j=1}^kf_j|^2}\nonumber\\
&\le 25\frac{(2^d - 1)2^{n(2d - 1)}}{2^{2nd}|\E \langle g,h_n\rangle\prod_{j=1}^kf_j|^2}\E|\langle g,h_n\rangle\prod_{j=1}^kf_j|^2 \nonumber \\
&\le 25(2^d - 1)\frac{2^{2n(k+1)/(2k+3)}\E |h_n|^2}{2^n|\E
\langle g,h_n\rangle\prod_{j=1}^kf_j|^2}\nonumber\\
&\le 25K(2^d - 1)\frac{\log^{d} n}{2^{n/(2k+3)}|\E
\langle g,h_n\rangle\prod_{j=1}^kf_j|^2}.
\end{align}

Let us also notice that for large $n$, by
(\ref{growth_condition}), Lemma \ref{calkowalnosc} and
(\ref{assumption})
\begin{align}\label{calka_obciecia}
|\E \langle &g,h_n\rangle\prod_{j=1}^k f_j| \ge |\E \langle g,h \rangle\prod_{j=1}^k f_j| - |\E
\langle g,h\rangle\ind{|h| >
2^n}\prod_{j=1}^{k}f_j| \nonumber\\
&\ge |\E \langle g,h\rangle \prod_{j=1}^k f_j| - 2^{n(k+1)/(2k+3)}K\frac{(\log
n)^{d}}{2^n} \ge \frac{5}{8}|\E \langle g,h \rangle\prod_{j=1}^k f_j| \ge
\frac{5}{8}.
\end{align}

Inequalities (\ref{Markow}), (\ref{czebyszew}) and
(\ref{calka_obciecia}) imply, that for large $n$ with probability
at least $0.9$ we have
\begin{align*}
&|\sum_{|\ii|\le 2^n}\langle g(\xdik), h(\xdi)\rangle\prod_{j=1}^kf_j(\xdij)| \nonumber\\
&\ge |\sum_{|\ii|\le 2^n} \langle g(\xdik),h_n(\xdi)\rangle\prod_{j=1}^k f_j(\xdij)| \\
&\phantom{xx}-
|\sum_{|\ii|\le 2^n}
\langle g(\xdik),h(\xdi)\rangle\ind{|h(\xdi)| > 2^n}\prod_{j=1}^kf_j(\xdij)| \nonumber\\
&\ge 2^{nd}\big(\frac{4}{5}|\E \langle g,h_n\rangle\prod_{j=1}^nf_j| -
\frac{1}{4}|\E
\langle g,h\rangle\prod_{j=1}^kf_j|\big) \nonumber\\
&\ge 2^{nd}\big(\frac{4}{5}\cdot\frac{5}{8}|\E \langle g,h\rangle\prod_{j=1}^nf_j|
- \frac{1}{4}|\E \langle g,h\rangle\prod_{j=1}^kf_j|\big)\ge\frac{2^{nd}}{4}|\E\langle g,
h\rangle\prod_{j=1}^kf_j|.
\end{align*}

Together with (\ref{oszacowanie_normy}) this yields that for large
$n$ with probability at least $0.8$,
\begin{displaymath}
\|(h_\ii)_{|\ii|\le 2^n}\|^\ast_{K,\mathcal{J},\log n} \ge
\frac{2^{nd/2}\log^{k/2}n}{4 \cdot 2^{d(k+1)/2}10^{(k+1)/2}}|\E
\langle g,h\rangle\prod_{j=1}^kf_j|.
\end{displaymath}

Thus, by Lemma \ref{chaos}, for large $n$
\begin{displaymath}
\p\big(\big|\sum_{|\ii|\le 2^n} \edi h(\xdi)\big| \ge c_d
\frac{2^{nd/2}\log^{k/2}n}{4 \cdot 2^{d(k+1)/2}10^{(k+1)/2}}|\E
\langle g,h\rangle\prod_{j=1}^kf_j|\big) \ge \frac{8}{10n},
\end{displaymath}
which together with (\ref{co_drugie_n}) gives
\begin{displaymath}
|\E \langle g,h\rangle\prod_{j=1}^kf_j| \le L_d C
\frac{4\cdot2^{d(k+1)/2}10^{(k+1)/2}}{c_d}\log^{(d-k)/2}n.
\end{displaymath}
In particular for sufficiently large $N$, for arbitrary functions
$g\colon \Sigma^{\#K} \to H$, $f_j \colon \Sigma^{\#J_j} \to \R$, $j=1,\ldots,k$, such that
\begin{align*}
\|g(\xdk)\|_\infty, \|f_j(X_{J_j})\|_2 &\le 1, \\
\|g(\xdk)\|_2, \|f_j(X_{J_j})\|_\infty &\le 2^{N/(2k+3)}
\end{align*}
we have
\begin{displaymath}
|\E \langle g,h\rangle\prod_{j=1}^kf_j| \le L_d C
\frac{4\cdot2^{d(k+1)/2}10^{(k+1)/2}}{c_d}\log^{(d-k)/2}n \le \tilde{L}_d
C \log^{(d-k)/2}N,
\end{displaymath}
which clearly implies (\ref{nasc}).
\end{proof}

\section{Sufficiency}

\begin{lemma} \label{szereg} Let $H = H(X_1,\ldots,X_d)$ be a nonnegative random variable, such that $\E H^2 < \infty$.
Then for $I \subseteq I_d$, $I\neq \emptyset$,$I_d$,
\begin{displaymath}
\sum_{l=0}^{\infty}\sum_{n=1}^\infty 2^{l+\#I^{c}n}
\p_{I^c}(\E_{I}H^2 \geq
2^{2l+\#I^{c}n})<\infty.
\end{displaymath}
\end{lemma}

\begin{proof}
\begin{align*}
\sum_l\sum_n 2^{l+\#I^{c}n} \p_{I^c}(\E_{I}H^2 \geq 2^{2l+\#I^{c}n}) &=
\sum_l 2^l \E_{I^c} \Big[\sum_n 2^{\#I^c n}\ind{\E_I|H|^2\ge 2^{2l + \#I^cn}}\Big]\\
&\le \sum_l 2^{1-l}\E_{I^c}\E_I H^2 \le 4\E H^2 < \infty.
\end{align*}
\end{proof}

\begin{lemma} \label{upper_bounds_only} Let $\mathbf{X} = (X_1,\ldots,X_d)$ and $\mathbf{\tilde{X}}(I) =
((X_i)_{i\in I}, (X^{(1)}_i)_{i \in I^c})$. Denote $H = |h|/(\llog |h|)^{d/2}$. If $\E |H|^2 < \infty$ and
$h_n = h\Ind{A_n}$, where
\begin{align*}
A_n \subseteq \big\{&x \colon  |h(x)|^2 \le 2^{nd}\log^d n\;
\textrm{and}\;\forall_{I\neq \emptyset,I_d}\; \E_I H^2 \le 2^{\#I^cn}\big\},
\end{align*}
then for $I \subseteq I_d$, $I \neq \emptyset$, we have
\begin{displaymath}
\sum_n \frac{2^{-n\#I}}{\log^{2d} n}\E[
|h_n(X)|^2|h_n(\tilde{X}(I))|^2] < \infty.
\end{displaymath}
\end{lemma}
\begin{proof}
\begin{itemize}
\item[a)] $I = I_d$
\begin{align*}
\sum_n \frac{\E |h_n|^4}{2^{nd}\log^{2d}n} &\le
\E |h|^4 \sum_n \frac{1}{2^{nd}\log^{2d}n}\ind{|h|^2 \le 2^{nd}\log^d n} \\
&\le L_d \E |h|^4\frac{1}{|h|^2(\llog |h|)^d} < \infty.
\end{align*}

\item[b)] $I \neq I_d,\emptyset$. Let us denote by $\E_I,
\E_{I^c}, \tilde{\E}_{I^c}$ respectively, the expectation with
respect to $(X_i)_{i\in I}$, $(X_i)_{i\in I^c}$ and
$(X_i^{(1)})_{i \in I^c}$. Let also $\tilde{h}$, $\tilde{h}_n$
stand for $h(\mathbf{\tilde{X}}(I))$,
$h_n(\mathbf{\tilde{X}}(I))$ respectively. Then

\begin{align*}
&\sum_n  \frac{\E (|h_n|^2\cdot |\tilde{h}_n|^2)}{2^{n\#I}\log^{2d}n}
\le 2\sum_n \frac{\E (|h_n|^2\cdot |\tilde{h}_n|^2\ind{|h| \le |\tilde{h}|})}{2^{n\#I}\log^{2d}n}\\
&\le 2\E \Big(|h|^2|\tilde{h}|^2\ind{|h|\le|\tilde{h}|}\\
&\times\sum_n
\frac{1}{2^{n\#I}\log^{2d}n }
\ind{\E_{I^c}|h|^2\ind{|h|^2\le 2^{2nd}}\le L_d 2^{\#In}\log^d n, \; |\tilde{h}|^2 \le 2^{2nd}}\Big)\\
&\le 2\E \Big(|h|^2|\tilde{h}|^2\ind{|h|\le|\tilde{h}|}\\
&\times\sum_n
\frac{1}{2^{n\#I}\log^{2d}n}
\ind{\E_{I^c}|h|^2\ind{|h|^2\le |\tilde{h}|^2}\le L_d 2^{\#In}\log^d n, \; |\tilde{h}|^2 \le 2^{2nd}}\Big)\\
&\le \tilde{L}_d\E \Big(|h|^2|\tilde{h}|^2\ind{|h|\le|\tilde{h}|}\frac{1}{(\E_{I^c}|h|^2\ind{|h|^2\le |\tilde{h}|^2})(\llog |\tilde{h}|)^d}\Big)\\
&= \tilde{L}_d\E_I \tilde{\E}_{I^c} \Big[|\tilde{h}|^2\E_{I^c}\Big(|h|^2\ind{|h|\le |\tilde{h}|}\frac{1}{(\E_{I^c}|h|^2\ind{|h|^2\le |\tilde{h}|^2})(\llog |\tilde{h}|)^d}\Big)\Big]\\
&\le \tilde{L}_d\E\frac{|\tilde{h}|^2}{(\llog |\tilde{h}|)^d} < \infty,
\end{align*}
where to obtain the second inequality, we used the fact that
\begin{align*}
\E_{I^c}& |h|^2\ind{|h|^2\le 2^{2nd}, \E_{I^c} H^2 \le 2^{\#In}} \\
&\le \E_{I^c} \frac{|h|^2}{(\llog |h|)^{d}}(\loglog 2^{nd})^d\ind{\E_{I^c} H^2 \le 2^{\#I n}}\\
&\le L_d \E_{I^c} H^2 \ind{\E_{I^c} H^2 \le 2^{\#In}} \log^{d} n \le L_d 2^{\#In}\log^d n.
\end{align*}
\end{itemize}
\end{proof}

\begin{lemma}\label{l4sequence3} Consider a square integrable, nonnegative random variable $Y$. Let $Y_n = Y\Ind{B_n}$, with
$B_n = \bigcup_{k\in K(n)}C_{k}$, where $C_0, C_1,C_2,\ldots$ are pairwise disjoint subsets of $\Omega$ and
\begin{align*}
K(n) = \{k \le n\colon\E (Y^2\Ind{C_k}) \le 2^{k-n}\}.
\end{align*}
Then
\begin{align*}
\sum_{n}
(\E Y_n^2)^2 < \infty
\end{align*}
\end{lemma}
\begin{proof}
Let us first notice that by the Schwarz inequality, we have
\begin{align*}
\Big(\sum_{k\in K(n)}\E (Y^2\Ind{C_k})\Big)^2 &= \Big((\sum_{k\in K(n)}2^{(n-k)/2}2^{(k-n)/2} \E (Y^2\Ind{C_k})\Big)^2\\
&\le\sum_{k\in K(n)}[2^{n-k} (\E (Y^2\Ind{C_k}))^2]\sum_{k\le n}2^{k-n} \\
&= 2\sum_{k\in K(n)}[2^{n-k} (\E (Y^2\Ind{C_k}))^2].
\end{align*}
Thus
\begin{align*}
\sum_{n}
(\E Y_n^2)^2&\le \sum_n 2\sum_{k\in K(n)}[2^{n-k} (\E (Y^2\Ind{C_k}))^2]\\
&\le 2 \sum_{k\colon \E (Y^2\Ind{C_k})> 0} (\E (Y^2\Ind{C_k}))^2\sum_{n\colon k\in K(n)}2^{n-k}\\
&\le 4\sum_{k\colon \E (Y^2\Ind{C_k})> 0} (\E (Y^2\Ind{C_k}))^2\max_{n\colon k\in K(n)}2^{n-k}\\
&\le4\sum_{k\colon \E (Y^2\Ind{C_k})> 0} (\E (Y^2\Ind{C_k}))^2\frac{1}{\E (Y^2\Ind{C_k})}\\
&\le4 \sum_k \E (Y^2\Ind{C_k})= 4 \E Y^2 < \infty.
\end{align*}
\end{proof}

\begin{proof}[Proof of sufficiency]
The proof consists of several truncation arguments.
The first part of it follows the proofs presented in \cite{GLKZ}
and \cite{AmcLat} for the real-valued case. Then some modifications
are required, reflecting the diminished integrability condition in the
Hilbert space case. At each step we will show that
\begin{equation}
\label{basic_series} \sum_{n=1}^\infty\p\big(\big|\sum_{|\ii|\le
2^n} \pi_d h_n(\xdi)\big| \ge C2^{nd/2}\log^{d/2} n\big) < \infty,
\end{equation}
with $h_n = h\Ind{A_n}$ for some sequence of sets $A_n$.
 In the whole proof we keep the notation $H = |h|/(\llog |h|)^{d/2}$.

Let us also fix $\eta_d \in(0,1)$, such that the following
implication holds
\begin{align}\label{implication}
\forall_{n=1,2,\ldots}\;|h|^2 \le \eta_d^2 2^{nd}\log^d n \implies H^2 \le 2^{nd}.
\end{align}
\paragraph{\textbf{Step 1}} Inequality (\ref{basic_series}) holds for any $C > 0$ if

\begin{displaymath}
A_n \subseteq \big\{x\colon |h(x)|^2 \ge \eta_d^2 2^{nd}\log^d n\big\}.
\end{displaymath}

\noindent We have, by the Chebyshev inequality and the inequality
$\E|\pi_d h_n| \le 2^d\E|h_n|$ (which follows directly from the
definition of $\pi_d$ or may be considered a trivial case of Lemma
\ref{moments_comparison}),

\begin{align*}
\sum_n \p\big(\big|\sum_{|\ii| \le 2^n}& \pi_d h_n(\xdi)\big|\ge C 2^{nd/2}\log^{d/2} n\big) \\
&\le \sum_{n=1}^\infty \frac{\E\big|\sum_{|\ii| \le 2^n} \pi_d h_n(\xdi)\big|}{C2^{nd/2}\log^{d/2}n}\\
&\le 2^d\sum_n \frac{2^{nd}\E|h|\ind{|h| > \eta_d 2^{nd/2}\log^{d/2}n}}{C2^{nd/2}\log^{d/2}n} \\
& =2^dC^{-1} \E\Big(|h| \sum_n \frac{2^{nd/2}}{\log^{d/2}n}\ind{|h| > \eta_d 2^{nd/2}\log^{d/2}n}\Big) \\
&\le L_dC^{-1}\E\frac{|h|^2}{(\llog |h|)^d}< \infty.
\end{align*}

\paragraph{\textbf{Step 2}}  Inequality (\ref{basic_series}) holds for any $C > 0$ if
\begin{displaymath}
A_n \subseteq \big\{x  \colon |h(x)|^2 \le
\eta_d^2 2^{nd}\log^d n, \,\exists_{I\neq\emptyset,I_d}
\;\E_I H^2 \ge 2^{\#I^cn}\big\}.
\end{displaymath}

As in the previous step, it is enough to prove that
\begin{displaymath}
\sum_{n=1}^\infty \frac{\E\big|\sum_{|\ii|\le 2^n} \edi
h_n(\xdi)\big|}{2^{nd/2}\log^{d/2} n} < \infty.
\end{displaymath}

The set $A_n$ can be written as
\begin{displaymath}
A_n = \bigcup_{I\subseteq I_d, I\neq I_d,\emptyset} A_n(I),
\end{displaymath}
where the sets $A_n(I)$ are pairwise disjoint and
\begin{displaymath}
A_{n}(I) \subseteq \{x\colon |h(x)|^2\leq 2^{2nd},\; \E_{I}H^2\geq
2^{\#I^{c}n}\}.
\end{displaymath}
Therefore it suffices to prove that
\begin{equation}\label{modified_second_step}
\sum_{n=1}^\infty \frac{\E\big|\sum_{|\ii|\le 2^n} \edi
h(\xdi)\Ind{A_n(I)}(\xdi)\big|}{2^{nd/2}\log^{d/2} n} < \infty.
\end{equation}
Let for $l\in \N$,
\begin{align*}
A_{n,l}(I):=\{&x\colon |h(x)|^2\leq 2^{2nd},\\
&2^{2l+2+\#I^{c}n}> \E_{I}H^2 \geq 2^{2l+\#I^{c}n}\}\cap A_n(I).
\end{align*}
Then $h_{n}\Ind{A_n(I)}=\sum_{l=0}^{\infty}h_{n,l}$, where
$h_{n,l}:=h_{n}\Ind{A_{n,l}(I)}$ (notice that the sum is actually
finite in each point $x\in \Sigma^d$ as for large $l$, $x \notin
A_{n,l}(I)$).

We have
\begin{align*}
\E|\sum_{|\ii|\leq 2^{n}}\edi h_{n,l}(\xdi)|&\leq
\sum_{|\ii_{I^c}|\le 2^n}\E_{I^{c}}\E_I|\sum_{|\ii_{I}|\leq
2^{n}}\mathbf{\epsilon}_{\ii_I}^{\rm{dec}} h_{n,l}(\xdi)|
\\
&\leq \sum_{|\ii_{I^c}|\le 2^n}\E_{I^{c}}
(\E_{I}|\sum_{|\ii_{I}|\leq 2^{n}}\epsilon^{\rm dec}_{\ii_{I}}
h_{n,l}(\xdi)|^{2})^{1/2}
\\
&\leq 2^{(\#I^{c}+\#I/2)n}\E_{I^{c}}(\E_{I}|h_{n,l}|^{2})^{1/2}
\\
&\leq L_d [2^{(\#I^{c}+d/2)n+l+1}\log^{d/2}n]\p_{I^c}(\E_{I}H^2\geq 2^{2l+\#I^{c}n}),
\end{align*}
where in the last inequality we used the estimate
\begin{align*}
\E_I h_{n,l}^2 \le& L_d\E_I[ (\log n)^d H^2\ind{2^{2l+2+\#I^{c}n}> \E_{I}H^2 \geq 2^{2l+\#I^{c}n}}]\\
\le& L_d 2^{2l+2+\#I^c n}(\log n) ^d \ind{\E_{I}H^2 \geq 2^{2l+\#I^{c}n}}.
\end{align*}
Therefore to get (\ref{modified_second_step}) it is enough to show
that
\begin{displaymath}
\sum_{l=0}^{\infty}\sum_{n}2^{l+\#I^{c}n}
\p_{I^c}(\E_{I}H^2 \geq 2^{2l+\#I^{c}n})<\infty.
\end{displaymath}
But this is just the statement of Lemma \ref{szereg}.

\paragraph{\textbf{Step 3}} Inequality (\ref{basic_series}) holds for any $C > 0$ if
\begin{align*}
A_n \subseteq \big\{&x \colon |h(x)|^2 \le \eta_d^2
2^{nd}\log^d n,\; \forall_{I\neq \emptyset,I_d} \;\E_I H^2 \le
2^{\#I^cn}\} \cap \bigcup_{I\subsetneq I_d }B_n^I,
\end{align*}
with $B_n^I = \bigcup_{k\in K(I,n)}C_{k}^I$ and
$C_0^I = \{x\colon \E_I H^2 \le 1\}$, $C_k^I = \{x\colon 2^{\#I^c(k-1)}< \E_I H^2 \le 2^{\#I^ck}\}$, $k\ge 1$,
$K(I,n) = \{k \le n\colon\E (H^2 \Ind{C_k^I})
\le 2^{k-n}\}$.

\noindent By Lemma \ref{moments_comparison} and the Chebyshev
inequality, it is enough to show that
\begin{displaymath}
\sum_n \frac{\E|\sum_{|\ii| \le 2^n}\edi
h_n(\xdi)|^4}{2^{2nd}\log^{2d} n} < \infty.
\end{displaymath}
The Khintchine inequality for Rademacher chaoses gives
\begin{align*}
L_d^{-1}\E|\sum_{|\ii| \le 2^n} \edi h_n(\xdi)|^4 &
\le \E(\sum_{|\ii| \le 2^n}|h_n(\xdi)|^2)^2 \\
&= \sum_{I \subseteq I_d} \sum_{|\ii| \le 2^n}
\sum_{\stackrel{|\jj| \le 2^n\colon}{\{k\colon i_k = j_k\}=I}}\E |[h_n(\xdi)|^2 |h_n(\mathbf{X_\jj^{dec}})|^2] \\
&\le \sum_{I \subseteq I_d} 2^{nd}2^{n(d-\#I)} \E[
|h_n(\mathbf{X})|^2 \cdot |h_n(\mathbf{\tilde{X}}(I))|^2],
\end{align*}
where $\mathbf{X} = (X_1,\ldots,X_d)$ and $\mathbf{\tilde{X}}(I) =
((X_i)_{i\in I}, (X^{(1)}_i)_{i \in I^c})$.

To prove the statement of this step it thus suffices to show that
for all $I\subseteq I_d$,
\begin{align}\label{repeated_argument}
S(I) := \sum_n \frac{2^{-n\#I}}{\log^{2d} n}\E[
|h_n(X)|^2|h_n(\tilde{X}(I))|^2] < \infty.
\end{align}
The case of nonempty $I$ follows from Lemma
\ref{upper_bounds_only}. It thus remains to consider the case
$I=\emptyset$. Set $H_I^2 = \E_I H^2$. We have
\begin{align*}
S(\emptyset) &= \sum_n \frac{(\E |h_n|^2)^2}{\log^{2d}n}= \sum_n (\E
(\frac{|h|^2}{\log^{d}n}\Ind{A_n}))^2 \le L_d \sum_n
(\E (H^2\Ind{A_n}))^2\\
&\le L_d \sum_n
(\E (H^2\sum_{I\subsetneq I_d} \Ind{B^I_n}))^2
\le \tilde{L}_d \sum_{I\subsetneq I_d} \sum_n(\E (H^2 \Ind{B^I_n}))^2\\
&=\tilde{L}_d \sum_{I\subsetneq I_d} \sum_n (\E (H_I^2\Ind{B_n^I}))^2 <
\infty
\end{align*}
by Lemma \ref{l4sequence3}, applied for $Y^2=\E_I H^2$, since $\E H_I^2 =
\E H^2< \infty$.

\paragraph{\textbf{Step 4}}
Inequality (\ref{basic_series}) holds for some $C \le L_d D$ if
\begin{align*}
A_n = \big\{&x \colon |h(x)|^2 \le \eta_d^2 2^{nd}\log^d n,\;
\forall_{I\neq \emptyset,I_d} \E_I H^2 \le 2^{\#I^cn}\} \cap \bigcap_{I\subsetneq I_d} (B_n^I)^c,
\end{align*}
where $B_n^I$ is defined as in the previous step.

Let us first estimate $\|(E_I |h_n|^2)^{1/2}\|_\infty$ for $I \subsetneq I_d$.
We have
\begin{align*}
\E_I |h_n|^2 &\le \E_I \Big[|h|^2\ind{|h|^2\le \eta_d 2^{nd}\log^d n}\sum_{k\le n, k \notin K(I,n)} \Ind{C_k^I}\Big]\\
&\le L_d \log^d n  \sum_{k\le n, k \notin K(I,n)} \Ind{C_k^I}\E_I H^2.
\end{align*}
The fact that we can restrict the summation to $k\le n$ follows directly from the definition
of $A_n$ for $I \neq \emptyset$ and for $I =\emptyset$ from (\ref{implication}).

The sets $C_k^I$ are pairwise disjoint and thus
\begin{align}\label{l_infinity}
\|\E_I |h_n|^2\|_\infty \le (L_d \log^d n)\max_{k\le n, k\notin K(I,n)} 2^{\#I^c k}
= L_d  2^{\#I^c \kimax(n)}\log^d n,
\end{align}
where
\begin{displaymath}
\kimax(n) = \max\{k\le n\colon k \notin K(I,n)\}.
\end{displaymath}
Therefore for $C > 0$,
\begin{align*}
&\sum_{n}
\exp\Big[-\frac{1}{L_d}\Big(\frac{C2^{nd/2}\log^{d/2} n}{2^{\#In/2}\|(\E_I |h_n|^2)^{1/2}\|_\infty}\Big)^{2/(d+\#I^c)}\Big]\\
&\le \sum_n\sum_{k\le n,k\notin K(I,n)}
\exp\Big[-\frac{1}{L_d}\Big(\frac{C2^{nd/2}\log^{d/2} n}{2^{\#In/2}2^{\#I^c k/2}\log^{d/2}n}\Big)^{2/(d+\#I^c)}\Big]\\
&= \sum_k \sum_{n\ge k,\; k \notin K(I,n)}
\exp\Big[-\frac{1}{\tilde{L}_d}\Big(C2^{\#I^c(n-k)/2}\Big)^{2/(d+\#I^c)}\Big].
\end{align*}

Notice that for each $k$ the inner series is bounded by a geometric series with the ratio smaller than some
$q_{d,C} < 1$ ($q_{d,C}$ depending only on $d$ and $C$). Therefore the right hand side of the above inequality is bounded by
\begin{align*}
K\sum_k \sup_{n\ge k,\; k \notin
K(I,n)}\exp\Big[-\frac{1}{\tilde{L}_d}\Big(C2^{\#I^c(n-k)/2}\Big)^{2/(d+\#I^c)}\Big],
\end{align*}
with the convention $\sup \emptyset = 0$. But $k \notin K(I,n)$
implies that $2^{\#I^c(n-k)/2} \ge (\E(
H^2\Ind{C^I_k}))^{-\#I^c/2}$. Therefore the above quantity is
further bounded by
\begin{align*}
K \sum_k \exp\Big[-\frac{1}{\tilde{L}_d}\Big(C^{-2/\#I^c}\E (H^2\Ind{C^I_k})\Big)^{-\#I^c/(d+\#I^c)}\Big] &\le
\bar{L}_d C^{-2/\#I^c}\sum_k \E (H^2\Ind{C_k^I})\\
&= \bar{L}_dC^{-2/\#I^c}\E H^2 < \infty,
\end{align*}
where we used the inequality
$e^x \ge c_dx^\alpha$ for all $x\ge 0$ and $0 \le \alpha \le 2d$.
We have thus proven that for all $I \subsetneq I_d$ and $C,L_d >
0$,
\begin{align}\label{l_infinity_series_estimation}
\sum_{n\colon A_n \neq \emptyset}
\exp\Big[-\frac{1}{L_d}\Big(\frac{C2^{nd/2}\log^{d/2} n}{2^{\#In/2}\|(\E_I |h_n|^2)^{1/2}\|_\infty}\Big)^{2/(d+\#I^c)}\Big] < \infty.
\end{align}

Now we will turn to the estimation of $\|h_n\|_{J_0,\mathcal{J}}$.
Let us consider $J_0\subseteq I_d$, $\mathcal{J} =
\{J_1,\ldots,J_l\} \in \mathcal{P}_{I_d\backslash J_0}$ and denote
as before $X = (X_1,\ldots,X_d)$, $X_{I} = (X_i)_{i\in I}$.
Recall that

\begin{align*}
\|h_n\|_{J_0,\mathcal{J}} = \sup\big\{&\E \langle h_n(X),
f_0(\xdjz)\rangle\prod_{i=1}^l f_i (X_{J_i}) \colon
\E|f_0(\xdjz)|^2\le
1,\\
&\E f_i^2(X_{J_i}) \le 1, \; i\ge 1 \big\}.
\end{align*}

In what follows, to simplify the already quite complicated
notation, let us suppress the arguments of all the functions and
write just $h$ instead of $h(X)$ and $f_i$ instead of
$f_i(X_{J_i})$.

Let us also remark that although $f_0$ plays special role in the
definition of $\|\cdot\|_{J_0,\mathcal{J}}$, in what follows  the
same arguments will apply to all $f_i$'s with the obvious use of
Schwarz inequality for the scalar product in $H$. We will
therefore not distinguish the case $i=0$ and $f_i^2$ will denote
either the usual power or $\langle f_0,f_0\rangle$, whereas
$\|f_i\|_2$ for $i=0$ will be the norm in $L^2(H,X_{J_0})$, which
may happen to be equal just $H$ if $J_0 = \emptyset$.

Since $\E |f_i|^2 \le 1$, $i=0,\ldots,l$, then for each $j =0,\ldots,l$ and $J \subsetneq J_j$ by the Schwarz inequality
applied conditionally to $X_{J_j\backslash J}$

\begin{align*}
\E |\langle &h_n,f_0\rangle \prod_{i=1}^l f_i\ind{\E_J f_j^2 >
a^2}| \\
&\le \E_{J_j\backslash J}
\big[(\E_{(J_j\backslash J)^c} \prod_{i=0}^l f_i^2)^{1/2}\ind{\E_J f_j^2 \ge a^2} (\E_{(J_j\backslash J)^c} |h_n|^2)^{1/2}\big]\\
&\le\E_{J_j\backslash J}\big[(\E_J f_j^2)^{1/2}\ind{\E_J f_j^2 \ge a^2} (\E_{(J_j\backslash J)^c} |h_n|^2)^{1/2} \big]\\
&\le L_d 2^{\kijmax(n)\#(J_j\backslash J)/2}\log^{d/2} n\E_{J_j\backslash J}[(\E_J f_j^2)^{1/2}\ind{\E_J f_j^2 \ge a^2}] \\
&\le L_d [2^{\kijmax(n)\#(J_j \backslash J)/2}\log^{d/2} n]a^{-1},
\end{align*}
where the third inequality follows from (\ref{l_infinity}) and the last one from the elementary fact $\E|X|\ind{|X|\ge a} \le a^{-1}\E|X|^2$.
This way we obtain
\begin{align}\label{truncation_in_J_norms}
&\|h_n \|_{J_0,\mathcal{J}} \\
&\le \sup\{\E[ \langle h_n,
f_0\rangle\prod_{i=1}^l f_i] \colon \|f_i\|_2 \le 1, \forall_{J \subsetneq J_i}\;\|(\E_J
f_i^2)^{1/2}\|_\infty
\le 2^{n\#(J_i\backslash J)/2} \}\nonumber\\
&+ L_d\sum_{i=0}^l \sum_{J\subsetneq J_i}2^{(\kiimax(n)-n)\#(J_i\backslash J)/2}\log^{d/2} n\nonumber \\
&\le \sup\{\E [\langle h_n,f_0\rangle\prod_{i=1}^l f_i] \colon
\|f_i\|_2 \le 1, \forall_{J \subsetneq J_i}\; \|(\E_J f_i^2)^{1/2}\|_\infty \le
2^{n\#(J_i\backslash J)/2} \}\nonumber\\
&+ L_d\sum_{I\subsetneq I_d}2^{(\kimax(n)-n)\#I^c/2}\log^{d/2}n \nonumber.
\end{align}

Let us thus consider arbitrary $f_i$, $i=0,\ldots,k$ such that
$\|f_i\|_2 \le 1$, $\|(\E_J f_i^2)^{1/2}\|_\infty \le
2^{n\#(J_i\backslash J)/2}$ for all $J \subsetneq J_i$ (note
that the latter condition means in particular that $\|f_i\|_\infty
\le 2^{n\#J_i/2}$).

We have by assumption (\ref{nasc}) for sufficiently large $n$,
\begin{equation*}
|\E [\langle h, f_0\rangle\prod_{i=1}^k f_i]| \le
\|h\|_{K,\mathcal{J},2^{nd/2}} \le L_d D\log^{(d - \deg\mathcal{J})/2}
n.
\end{equation*}

We have also
\begin{align*}
\E|\langle h,f_0\rangle\ind{|h|^2 \ge
\eta_d2^{nd}\log^{d}n}\prod_{i=1}^kf_i|
&\le \E[|h|\ind{|h|^2 \ge \eta_d 2^{nd}\log^d n}]\prod_{i=0}^k\|f_i\|_\infty \\
&\le 2^{nd/2}\E[|h|\ind{|h|^2 \ge \eta_d 2^{nd}\log^{d}n}] =: \alpha_n.
\end{align*}

Also for $I \subseteq I_d, I\neq \emptyset, I_d$, denoting
$\tilde{h}_n = h\ind{|h|^2\le \eta_d 2^{nd}\log^{d}n}$, we get
\begin{align*}
\E|\langle \tilde{h}_n, &f_0\rangle\prod_{i=1}^k f_i \ind{\E_I
H^2 \ge 2^{n\#I^c}}|\\
&\le \E_{I^c}\big[(\E_I |\tilde{h}_n|^2)^{1/2}\ind{\E_I H^2 \ge 2^{n\#I^c}}\prod_{i=0}^k (\E_{J_i \cap I} |f_i|^2)^{1/2}\big]\nonumber\\
&\le [\prod_{i=0}^l 2^{n\#(J_i\cap I^c)/2}]
\E_{I^c}[(\E_I |\tilde{h}_n|^2)^{1/2}\ind{\E_I H^2 \ge 2^{n\#I^c}}]\nonumber\\
&\le L_d 2^{n\#I^c/2} \E_{I^c}[(\E_I H^2\log^d n)^{1/2}\ind{\E_I H^2 \ge 2^{n\#I^c}}]\\
&\le L_d [2^{n\#I^c/2}\log^{d/2} n]\E_{I^c}[(\E_I H^2)^{1/2}\ind{\E_I H^2 \ge 2^{n\#I^c}}]=:\beta^I_n.
\end{align*}

Let us denote $\bar{h}_n = \tilde{h}_n\prod_{\emptyset\neq I\subsetneq I_d}\ind{\E_I H^2 \le 2^{\#I^c n}}$ and
$\gamma_n^I = \E|\bar{h}_n\Ind{B^I_n}|^2$.
Combining the three last inequalities we obtain
\begin{align*}
|\E \langle h_n,f_0\rangle \prod_{i=1}^l f_i| \le& |\E \langle h,f_0\rangle \prod_{i=0}^l f_i|
+ |\E \langle h_n\Ind{A_n^c},f_0 \rangle \prod_{i=1}^l f_i|\\
\le& L_dD\log^{(d-\deg\mathcal{J})/2}n + \E|\langle h\ind{|h|^2\ge 2^{nd}\log^{d}n},f_0\rangle\prod_{i=1}^lf_i|\nonumber\\
&+\sum_{\emptyset \neq I\subsetneq I_d}\E|\langle \tilde{h}_n\ind{\E_I H^2 \ge 2^{n\#I^c}}, f_0\rangle \prod_{i=1}^l f_i|\nonumber\\
&+ \sum_{I\subsetneq I_d}\E|\langle\bar{h}_n\Ind{B^I_n},f_0\rangle\prod_{i=1}^l f_i|\nonumber\\
\le& L_dD\log^{(d-\deg\mathcal{J})/2}n + \alpha_n + \sum_{\emptyset\neq I\subsetneq I_d}\beta_n^I +
\sum_{I\subsetneq I_d}\sqrt{\gamma_n^I}\nonumber.
\end{align*}

Now, combining the above estimate with (\ref{truncation_in_J_norms}), we obtain
\begin{align}\label{J_norm_estimate_nowe}
\|h_n\|_{J_0,\mathcal{J}} &\le L_d\sum_{I\subsetneq I_d} 2^{(\kimax(n)-n)\#I^c/2}\log^{d/2}n+
L_dD\log^{(d-\deg\mathcal{J})/2}n\\
& + \alpha_n + \sum_{\emptyset\neq I\subsetneq I_d}\beta_n^I +
\sum_{I\subsetneq I_d}\sqrt{\gamma_n^I}.\nonumber
\end{align}

Let us notice that
\begin{align}\label{alphabetagammaseries}
\sum_n \frac{\alpha_n}{\log^{d/2} n} &< \infty, \nonumber\\
\forall_{I\neq\emptyset, I_d}\sum_n \frac{\beta_n^I}{\log^{d/2} n} & < \infty,\\
\forall_{I\neq\emptyset, I_d} \nonumber \sum_n \frac{(\gamma_n^I)^2}{\log^{2d} n} &< \infty.
\end{align}

The first inequality was proved in Step 1. The proof of the second one is straightforward. Indeed, we have
\begin{align*}
\sum_n \frac{\beta_n^I}{\log^{d/2} n} &= L_d\E_{I^c} [(\E_I H^2)^{1/2}\sum_n 2^{n\#I^c/2}\ind{\E_I H^2 \ge 2^{n\#I^c}}]\\
&\le \tilde{L}_d \E_{I^c} \E_I H^2 = \tilde{L}_d \E H^2 < \infty.
\end{align*}
The third inequality is implicitly proved in Step 3. Let us however present an explicit argument.
\begin{align*}
\sum_n \frac{(\gamma_n^I)^2}{\log^{2d} n} & \le \sum_n \Bigg(\E\frac{|h|^2\ind{|h|\le \eta_d 2^{nd/2}\log^{d/2}n}\Ind{B_n^I}}{\log^d n}\Bigg)^2\\
&\le L_d \sum_n (\E_{I^c} \E_I (H^2\Ind{B_n^I}))^2 < \infty
\end{align*}
by Lemma \ref{l4sequence3} applied to the random variable
$\sqrt{\E_I H^2}$.

We are now in position to finish the proof. Let us notice that we have either
$\E( |h|^2\ind{|h|^2\le 2^{2nd}}) \le 1$, or we can use the function
\begin{displaymath}
g = \frac{h\ind{\|h|^2\le 2^{2nd}}}{(\E (|h|^2\ind{|h|^2\le 2^{2nd}}))^{1/2}}
\end{displaymath}
as a test function in the definition of $\|h\|_{I_d,\emptyset,2^{nd}}$, obtaining
\begin{displaymath}
(\E (|h|^2\ind{|h|^2 \le 2^{2nd}}))^{1/2} = \E\langle h,g\rangle \le \|h\|_{I_d,\emptyset,2^{nd}} <D \log^d n
\end{displaymath}
for large $n$. Combining this estimate with
Corollary \ref{tail_Hoeffding},
 we can now write
\begin{align}\label{final_series}
&\sum_{n}\p\big(|\sum_{|\ii|\le 2^n} \pi_d h_n(\xdi)| \ge \tilde{L}_d(D+C)2^{nd/2}\log^{d/2}n\big)\\
&\le \tilde{L}_d \sum_{J_0\subsetneq I_d}\sum_{\mathcal{J}\in
\mathcal{P}_{I_d\backslash J_0}}\sum_n
\exp\Big[-\frac{1}{\tilde{L}_d}\Big(\frac{C2^{nd/2}\log^{d/2}n}{2^{nd/2}\|h_n\|_{J_o,\mathcal{J}}}\Big)^{2/\deg\mathcal{J}}\Big]\nonumber\\
&+ \tilde{L}_d\sum_{I\subsetneq I_d}\sum_n
\exp\Big[-\frac{1}{\tilde{L}_d}
\Big(\frac{C2^{nd/2}\log^{d/2}n}{2^{n\#I/2}\|(\E_I|h_n|^2)^{1/2}\|_\infty}\Big)^{2/(d+\#I^c)}\Big]\nonumber.
\end{align}
The second series is convergent by (\ref{l_infinity_series_estimation}).

Thus it remains to prove the convergence
of the first series. By  (\ref{J_norm_estimate_nowe}), we have for all $J_0, \mathcal{J}$
\begin{align*}
&\exp\Big[-\frac{1}{\tilde{L}_d}\Big(\frac{C\log^{d/2}n}{\|h_n\|_{J_o,\mathcal{J}}}\Big)^{2/\deg\mathcal{J}}\Big]\\
&\le \sum_{I\subsetneq I_d}\exp\Big[-\frac{1}{L_d}\Big(\frac{C\log^{d/2}n}{2^{(\kimax(n)-n)\#I^c/2}\log^{d/2}n}\Big)^{2/\deg\mathcal{J}}\Big]\\
&+ \exp\Big[-\frac{1}{L_d}\Big(\frac{C\log^{d/2}n}{D\log^{(d-\deg\mathcal{J})/2}n}\Big)^{2/\deg\mathcal{J}}\Big]
+ \exp\Big[-\frac{1}{L_d}\Big(\frac{C\log^{d/2}n}{\alpha_n}\Big)^{2/\deg\mathcal{J}}\Big]\\
&+\exp\Big[-\frac{1}{L_d}\Big(\frac{C\log^{d/2}n}{\sum_{\emptyset\neq I\subsetneq I_d}\beta_n^I}\Big)^{2/\deg\mathcal{J}}\Big]
+\exp\Big[-\frac{1}{L_d}\Big(\frac{C\log^{d/2}n}{\sum_{I\subsetneq I_d}\sqrt{\gamma_n^I}}\Big)^{2/\deg\mathcal{J}}\Big],
\end{align*}
(under our permanent convention that the values of $L_d$ in different equations need not be the same).
The series determined by the three last components at the right-hand side are convergent by (\ref{alphabetagammaseries})
since $e^{-x} \le L_r x^{-r}$ for $r> 0$. The series corresponding to the second component is convergent for
$C$ large enough and we can take $C = L_d D$. As for the series corresponding to the first term, we have, just as in
the proof of (\ref{l_infinity_series_estimation}) for any $I\subsetneq I_d$,
\begin{align*}
&\sum_n \exp\Big[-\frac{1}{L_d}\Big(\frac{C\log^{d/2}n}{L_d2^{(\kimax(n)-n)\#I^c/2}\log^{d/2}n}\Big)^{2/\deg\mathcal{J}}\Big]\\
&\le \sum_k\sum_{n\ge k, k\notin K(I,n)} \exp\Big[-\frac{1}{\tilde{L_d}}\Big(C2^{(n-k)\#I^c/2}\Big)^{2/\deg\mathcal{J}}\Big]\\
& \le K\sum_k \sup_{n\ge k, k\notin K(I,n)}\exp\Big[-\frac{1}{\tilde{L}_d}\Big(C2^{(n-k)\#I^c/2}\Big)^{2/\deg\mathcal{J}}\Big]\\
&\le \bar{K}\sum_k \E( H^2\Ind{C_k^I}) = \bar{K}\E H^2 < \infty.
\end{align*}
We have thus proven the convergence of the series at the left-hand side of (\ref{final_series}) with $C \le L_dD$,
which ends Step 5.

\paragraph{}
Now to finish the proof, we just split $\Sigma^d$ for each $n$ into four sets, described by steps 1--4 and use the triangle inequality, to show
that
\begin{displaymath}
\sum_n \p\Big(\Big|\sum_{|\ii|\le 2^n} h(\xdi)| \ge L_d D 2^{nd/2}\log^{d/2} n\Big|\Big) < \infty,
\end{displaymath}
which proves the sufficiency part of the theorem by Corollary \ref{from_series_to_lil}.
\end{proof}

\section{The undecoupled case}
\begin{theorem}\label{undecoupled_lil_theorem}
For any function $h\colon \Sigma^d \to H$ and a sequence $X_1,X_2,\ldots$ of $i.i.d.$, $\Sigma$-valued random
variables, the LIL (\ref{lil}) holds if and only if $h$
\begin{displaymath}
\E\frac{|h|^2}{(\llog |h|)^d} < \infty,
\end{displaymath}
$h$ is completely degenerate for the law of $X_1$ and the growth conditions
(\ref{nasc}) are satisfied.

More precisely, if (\ref{lil}) holds, then (\ref{nasc}) is satisfied with $D = L_d C$ and
conversely, (\ref{nasc}) together with complete degeneration and the integrability condition imply (\ref{lil}) with $C = L_d D$.
\end{theorem}
\begin{proof}

Sufficiency follows from Corollary \ref{lil_dec_diagLIL} and Theorem \ref{decoupled_lil_theorem}.
To prove the necessity assume
that (\ref{lil}) holds and observe that from Lemma \ref{GineZhang}
and Corollary \ref{randomized_decoupled_series}, $h$ satisfies the
randomized decoupled LIL (\ref{randomized_decoupled_lil}) and
thus, by Theorem \ref{decoupled_lil_theorem}, (\ref{nasc_integrability}) holds and the growth
conditions (\ref{nasc}) on functions $\|h\|_{K,\mathcal{J},u}$ are satisfied
(note that the $\|\cdot\|_{\mathcal{J},u}$ norms of the kernel
$h(X_1,\ldots,X_d)$ and $\varepsilon_1\cdots\varepsilon_d
h(X_1,\ldots,X_d)$ are equal). The complete degeneracy of $\langle \varphi, h\rangle$ for any $\varphi \in H$
follows from the necessary conditions for real-valued kernels. Since by (\ref{nasc_integrability}), $\E_i h$ is well defined in
the Bochner sense, we must have $\E_i h = 0$.
\end{proof}

\end{document}